\documentclass[12pt]{amsart}
\usepackage[T1]{fontenc}
\usepackage[utf8]{inputenc}
\usepackage{lmodern}
\usepackage{xcolor}
\usepackage{mathrsfs,amssymb,amsmath,amsbsy,mathtools}
\usepackage{enumitem,needspace,etoolbox}
\usepackage[a4paper,textwidth=16cm,textheight=23cm,centering]{geometry}
\usepackage[expansion=false]{microtype}
\usepackage[colorlinks=true,linkcolor=blue,citecolor=blue,urlcolor=blue,
  bookmarksopen=true,bookmarksnumbered=true]{hyperref}
\allowdisplaybreaks[2]

\newtheorem{theorem}{Theorem}[section]
\newtheorem{lemma}[theorem]{Lemma}

\newtheorem{proposition}[theorem]{Proposition}
\theoremstyle{definition}

\theoremstyle{remark}
\newtheorem{remark}[theorem]{Remark}
\numberwithin{equation}{section}

\newcommand{\R}{\mathbb R}
\newcommand{\T}{\mathbb T}
\newcommand{\dd}{\,\mathrm d}
\newcommand{\Vq}{\widetilde V_q}
\newcommand{\Wulff}[1]{\left[#1\right]}

\makeatletter

\renewcommand{\@settitle}{%
  \begin{center}
    \normalfont\large\bfseries\boldmath\color{black}%
    \@title\par
  \end{center}%
}
\patchcmd{\@maketitle}{\global\topskip42\p@}
  {\global\topskip12\p@}
  {}{\ClassError{amsart}{Unable to adjust the title spacing}{}}
\patchcmd{\@maketitle}{\dimen@34\p@}
  {\dimen@18\p@}
  {}{\ClassError{amsart}{Unable to adjust the abstract spacing}{}}

\newcommand{\@titleaddresses}{%
  \par\nobreak\smallskip
  \begingroup
    \normalfont\footnotesize
    \def\author##1{}%
    \def\address##1##2{\par\noindent ##2\par}%
    \def\curraddr##1##2{}%
    \def\email##1##2{}%
    \def\urladdr##1##2{}%
    \addresses
  \endgroup
  \global\let\addresses\@empty
}
\patchcmd{\@setauthors}{\centering\footnotesize}
  {\centering\normalsize}
  {}{\ClassError{amsart}{Unable to adjust the author font}{}}
\patchcmd{\@setauthors}{\MakeUppercase{\authors}}
  {\authors}
  {}{\ClassError{amsart}{Unable to adjust the author name}{}}
\patchcmd{\@setauthors}{\endtrivlist}
  {\@titleaddresses\endtrivlist}
  {}{\ClassError{amsart}{Unable to move the affiliation}{}}

\renewcommand{\@adminfootnotes}{%
  \let\@makefnmark\relax
  \let\@thefnmark\relax
  \@footnotetext{%
    \hypersetup{urlcolor=black}%
    \urlstyle{same}%
    \begingroup
      \def\author##1{}%
      \def\address##1##2{}%
      \def\curraddr##1##2{}%
      \def\urladdr##1##2{}%
      \def\email##1##2{%
        Email:\enspace\href{mailto:##2}{\nolinkurl{##2}}.\par}%
      \addresses
    \endgroup
    \ifx\@empty\thankses\else\@setthanks\fi
  }%
}
\patchcmd{\footnoterule}{\@width 5pc}{\@width 3pc}{}{}

\renewcommand{\@setsubjclass}{%
  {\bfseries\subjclassname:}\enspace\@subjclass}
\renewcommand{\@setkeywords}{%
  {\bfseries Keywords:}\enspace\@keywords}
\renewenvironment{abstract}{%
  \global\setbox\abstractbox=\vtop\bgroup
    \normalfont\normalsize
    \list{}{%
      \labelwidth\z@
      \leftmargin1pc \rightmargin\leftmargin
      \listparindent\normalparindent \itemindent\z@
      \parsep\z@ \@plus\p@
      
    }%
    \item[]\textbf{Abstract}\enspace\ignorespaces
}{%
    \par\smallskip
    \ifx\@empty\@subjclass\else
      {\raggedright\noindent\@setsubjclass\par}\smallskip
    \fi
    \ifx\@empty\@keywords\else
      {\raggedright\noindent\@setkeywords\par}
    \fi
    \endlist\egroup
    \ifx\@setabstract\relax\@setabstracta\fi
}

\renewcommand{\@secnumfont}{\bfseries}
\renewcommand{\section}{\@startsection{section}{1}%
  \z@{.7\linespacing\@plus\linespacing}{.5\linespacing}%
  {\normalfont\bfseries\boldmath\color{black}\centering}}
\renewcommand{\subsection}{\@startsection{subsection}{2}%
  \z@{.5\linespacing\@plus.7\linespacing}{-.5em}%
  {\normalfont\bfseries\boldmath\color{black}}}
\renewcommand{\subsubsection}{\@startsection{subsubsection}{3}%
  \z@{.5\linespacing\@plus.7\linespacing}{-.5em}%
  {\normalfont\bfseries\boldmath\color{black}}}
\makeatother

\title[$L_p$ inequalities for planar dual quermassintegrals]
{On the $L_p$ Brunn--Minkowski inequality\\
for planar dual quermassintegrals}
\author{Jiangyan Tao}
\address{School of Mathematical Sciences, Zhejiang Normal University,
Zhejiang 321004, China}
\email{taojiangyan@zjnu.edu.cn}
\subjclass[2020]{52A40, 52A20, 26D15}
\keywords{$L_p$ addition, dual quermassintegral, Brunn--Minkowski inequality,
weighted Poincar\'e inequality}
\date{}

\begin{document}

\begin{abstract}
We determine the sharp $L_p$ Brunn--Minkowski range for planar dual
quermassintegrals on origin-symmetric convex bodies: $0<q\leq2$ for
$0\leq p<1$, and $0<q\leq4$ for $p=1$. The main new results are the
Minkowski inequality for $2<q<4$, proved using a weighted spectral
inequality on the projective circle, and parallelogram counterexamples
for $0<p<1$ and $q>2$. Together with the known logarithmic and endpoint
results, we obtain the complete equality classification: only positive
dilates occur, except at $(p,q)=(0,2)$, where parallelograms with
corresponding sides parallel also give equality.
\end{abstract}

\maketitle

\section{Introduction}\label{sec:intro}

The Brunn--Minkowski inequality expresses the concavity of the
$n$-th root of volume under Minkowski addition and is a basic principle
of convex geometry; see \cite{Schneider}. Two natural directions of
extension are to replace volume by another geometric functional and to
replace Minkowski addition by an interpolation of support functions.
Firey's $L_p$ addition \cite{Firey1962} and the logarithmic and $L_p$
inequalities of B\"or\"oczky, Lutwak, Yang, and Zhang \cite{BLYZ} provide
such interpolations. In the range $0\leq p<1$, the origin-symmetry
assumption plays a central role. The present paper concerns these
interpolations for dual quermassintegrals in the plane.

Dual quermassintegrals belong to the dual Brunn--Minkowski theory
initiated by Lutwak \cite{Lutwak1975}. They are defined through radial
functions rather than support functions. Their variational properties
are closely related to the dual curvature measures and Minkowski
problems introduced by Huang, Lutwak, Yang, and Zhang \cite{HLYZ};
see \cite{HYZ} for a survey. Concavity inequalities for these functionals
are therefore relevant both to geometric inequalities and to the
associated uniqueness questions. Here the interpolation is taken in
support-function coordinates, not by radial addition.

We write $\mathcal K_e^2$ for the origin-symmetric convex bodies in
$\R^2$, with nonempty interior, and denote their support and radial
functions by $h_K$ and $\rho_K$. For $q>0$, we use the normalization
\begin{equation}\label{eq:def-vq}
 \Vq(K)=\frac1q\int_{\mathbb S^1}\rho_K(u)^q\,\dd u
       =\int_K|x|^{q-2}\,\dd x.
\end{equation}
Thus $\widetilde V_2$ is area, while $\widetilde V_4$ is the polar moment
of inertia about the origin. For a positive continuous even function
$\varphi$ on $\mathbb S^1$, let
\[
 \Wulff{\varphi}
 =\{x\in\R^2:\langle x,u\rangle\leq\varphi(u)
       \text{ for every }u\in\mathbb S^1\}
\]
be its Wulff shape. For $K,L\in\mathcal K_e^2$ and $0\leq\lambda\leq1$,
define
\begin{equation}\label{eq:def-lp-sum}
 (1-\lambda)\mathbin{\cdot}K+_p\lambda\mathbin{\cdot}L
 =\Wulff{\bigl((1-\lambda)h_K^p+\lambda h_L^p\bigr)^{1/p}},
 \qquad 0<p\leq1,
\end{equation}
and, at $p=0$, define
\begin{equation}\label{eq:def-l0-sum}
 (1-\lambda)\mathbin{\cdot}K+_0\lambda\mathbin{\cdot}L
 =\Wulff{h_K^{1-\lambda}h_L^\lambda}.
\end{equation}
For $p=1$, this is the usual Minkowski combination $(1-\lambda)K+\lambda L$.

The problem is to determine the parameters for which
\begin{equation}\label{eq:main-bm}
 \Vq\bigl((1-\lambda)\mathbin{\cdot}K+_p
             \lambda\mathbin{\cdot}L\bigr)^{p/q}
 \geq (1-\lambda)\Vq(K)^{p/q}+\lambda\Vq(L)^{p/q},
 \qquad p>0,
\end{equation}
holds for all $K,L\in\mathcal K_e^2$. Its logarithmic form is
\begin{equation}\label{eq:main-l0-bm}
 \Vq\bigl((1-\lambda)\mathbin{\cdot}K+_0
             \lambda\mathbin{\cdot}L\bigr)
 \geq \Vq(K)^{1-\lambda}\Vq(L)^\lambda.
\end{equation}
The exponent $p/q$ is chosen so that positive dilates give equality. When $q=2$,
these are the planar $L_p$ and log-Brunn--Minkowski inequalities for area.

For Minkowski addition, Xi and Zhang \cite{XiZhang} established the
inequality for $0<q\leq1$ without symmetry assumptions, as part of their
$L_p$ theory in the range $p\geq q$. Sadovsky and Zhang
\cite{SadovskyZhang} subsequently proved the inequality for
origin-symmetric convex bodies in $\R^n$ throughout $0<q\leq n$.
For larger indices, Li, Wan, and Zhang \cite{LiWanZhang} obtained
counterexamples for $q>n+2$ and established the endpoint $q=n+2$,
including its equality condition, using Hadwiger's inequality for the
polar moment of inertia \cite{Hadwiger1956}. They conjectured validity,
with equality only for dilates, in the remaining interval $n<q<n+2$.
In dimension two, this leaves the range $2<q<4$ to be settled.

The logarithmic case has a different threshold. Xiong and Yang
\cite{XY} proved \eqref{eq:main-l0-bm} for $0<q<2$ in the plane, with
equality only for positive dilates, and constructed unconditional
rectangle counterexamples for $q>2$. At $q=2$, the theorem of
B\"or\"oczky, Lutwak, Yang, and Zhang \cite{BLYZ} applies; its equality
cases also include parallelograms with corresponding sides parallel.
For $0<p<1$, Chen, Tang, and Wang \cite{CTW} obtained near-ball results
for dual quermassintegrals and related uniqueness theorems, and studied
radially weighted extensions in \cite{CTWweighted}. Hu and Ivaki
\cite{HI} proved $L_p$ inequalities beyond the dimensional index in the
unconditional class, using weighted centro-affine Poincar\'e inequalities.
Li, Wan, and Zhang \cite{LiWanZhang} also established the Minkowski
inequality for unconditional bodies throughout $0<q\leq n+1$.
These results make it important to distinguish unconditional bodies
from general origin-symmetric bodies when seeking a sharp range.

Our main theorem gives the complete planar classification.

\begin{theorem}\label{thm:sharp-threshold}
Let $0\leq p\leq1$ and $q>0$. Then \eqref{eq:main-bm} when $p>0$,
and \eqref{eq:main-l0-bm} when $p=0$, hold for every
$K,L\in\mathcal K_e^2$ and every $0\leq\lambda\leq1$ if and only if
\begin{equation}\label{eq:sharp-region}
 \begin{cases}
  0<q\leq2,&0\leq p<1,\\
  0<q\leq4,&p=1.
 \end{cases}
\end{equation}
In this range, for $0<\lambda<1$, equality holds if and only if
$L=cK$ for some $c>0$, except when $(p,q)=(0,2)$, where equality also
holds for origin-symmetric parallelograms with corresponding sides
parallel. Outside this range, there are strict counterexamples at
$\lambda=1/2$.
\end{theorem}

We prove the inequality and classify equality for $p=1$, $2<q<4$,
resolving the planar conjecture of Li, Wan, and Zhang
\cite{LiWanZhang}. We also construct counterexamples for every
$0<p<1$, $q>2$. Thus the upper threshold remains $2$ for every
$p<1$ and jumps to $4$ at $p=1$.

For $2<q<4$, the second variation of $\Vq^{1/q}$, expressed in
radial and polar coordinates, reduces concavity to a weighted
Poincar\'e inequality on the projective circle. In the spirit of
\cite{KM,HI}, we use the planar support-function condition to obtain
the required spectral estimate. A periodic primitive leads to the
weight $v^{-\alpha}$, where $\alpha=q-2\in(0,2)$ and $v$ is a positive
support function. Comparison with solutions of $s''+s=0$ on short
nodal intervals gives a strict gap, which compactness makes uniform
under common inradius and circumradius bounds. This gap preserves
equality information through smooth approximation: a Green-kernel
identity forces the variance to vanish, and weak continuity of
surface area measures, together with the support-function equation
on complementary arcs, yields positive dilates.

For $0<p<1$ and $0<q\leq2$, normalization and comparison of means
reduce the inequality to the logarithmic case. Sharpness follows
from parallelograms with nearly parallel generating directions,
whose $L_p$ midpoint is identified exactly. In the thin limit,
the moment loss has order $t^q$, while the midpoint scaling has
a quadratic loss, giving strict counterexamples for $q>2$.
These bodies are not unconditional, consistent with the larger
ranges under that symmetry. At $q=4$, we give a self-contained
planar moment-lifting proof of the known endpoint
\cite{Hadwiger1956,LiWanZhang}. The rectangle calculation in
\cite[Section~4.1]{LiWanZhang} gives failure for $q>4$.

Section~\ref{sec:prelim} collects the required facts about support
functions, Wulff shapes, and approximation. Sections~\ref{sec:p0}
and~\ref{sec:positive-p} treat $p=0$ and $0<p<1$, respectively.
Section~\ref{sec:p1} proves the spectral estimates and the Minkowski
inequality for $2<q<4$. Section~\ref{sec:endpoint} treats the endpoint
and sharpness, and completes the proof of
Theorem~\ref{thm:sharp-threshold}.

\vskip20pt
\section{Preliminaries}\label{sec:prelim}

We use the standard notation of convex geometry; see
\cite{Gardner,Schneider}. The Euclidean inner product and norm are
$\langle\cdot,\cdot\rangle$ and $|\cdot|$. We write $B^2$ for the closed
unit disk, $\mathbb S^1$ for its boundary, and $|K|$ for planar area.
A body is unconditional if it is invariant under reflection in each
coordinate axis. A pair of unconditional bodies is always understood
to use the same orthonormal coordinate system.

For a convex body $K$ containing the origin in its interior,
\[
 h_K(u)=\max_{x\in K}\langle x,u\rangle,
 \qquad
 \rho_K(u)=\max\{r\geq0:ru\in K\},
 \qquad u\in\mathbb S^1.
\]
Its polar body is
\[
 K^\circ=\{y\in\R^2:\langle x,y\rangle\leq1
                  \text{ for every }x\in K\},
\]
and the support and radial functions satisfy
\begin{equation}\label{eq:polar-support-radial}
 h_{K^\circ}(u)=\rho_K(u)^{-1}.
\end{equation}
For $a,b\geq0$, Minkowski addition satisfies
$h_{aK+bL}=ah_K+bh_L$. Moreover,
\[
 d_H(K,L)=\|h_K-h_L\|_\infty,
 \qquad
 K\subset L\quad\Longleftrightarrow\quad h_K\leq h_L.
\]
If $rB^2\subset K\subset RB^2$ for $0<r<R$, then both $h_K$ and
$\rho_K$ lie between $r$ and $R$. Such support functions have a common
Lipschitz bound, and the family is compact in the uniform topology.
Within this family, Hausdorff convergence also implies uniform
convergence of radial functions.

For positive continuous data, the Wulff construction in the
introduction is monotone, homogeneous, and continuous under uniform
convergence. In particular,
\[
 \varphi\leq\psi\ \Longrightarrow\ \Wulff{\varphi}\subset\Wulff{\psi},
 \qquad
 \Wulff{c\varphi}=c\Wulff{\varphi}\quad(c>0).
\]
The defining function need not be a support function: in general,
$h_{\Wulff{\varphi}}\leq\varphi$. This distinction is relevant to
\eqref{eq:def-lp-sum} for $p<1$. To see continuity, if
$\varphi_j\to\varphi>0$ uniformly, then for every sufficiently small
$\varepsilon>0$ and all large $j$,
\[
 (1-\varepsilon)\Wulff{\varphi}
 \subset\Wulff{\varphi_j}
 \subset(1+\varepsilon)\Wulff{\varphi}.
\]
Consequently, all the combinations in \eqref{eq:def-lp-sum} and
\eqref{eq:def-l0-sum} are continuous on $\mathcal K_e^2$.

The equality of the two expressions in \eqref{eq:def-vq} follows by
polar integration. The standard planar normalization
$\frac12\int_{\mathbb S^1}\rho_K^q\,\dd u$ equals
$\frac q2\Vq(K)$; this fixed factor does not change the inequalities
or their equality cases. For $q>0$, the density $|x|^{q-2}$ is locally
integrable, and
\begin{equation}\label{eq:vq-homogeneous}
 \Vq(cK)=c^q\Vq(K),\qquad c>0.
\end{equation}
The functional is invariant under orthogonal transformations and is
strictly increasing under proper inclusion of convex bodies with
nonempty interior. It is also continuous under Hausdorff convergence
in $\mathcal K_e^2$, by uniform convergence of radial functions.
These properties will be used without further mention.

We identify functions on $\mathbb S^1$ with their $2\pi$-periodic
angular representations. Origin symmetry makes support functions
$\pi$-periodic, so we also use the projective circle
$\T_\pi=\R/\pi\mathbb Z$. Integration on $\T_\pi$ means integration
over one period of length $\pi$, without normalization. The space
$H^1(\T_\pi)$ consists of square-integrable $\pi$-periodic functions
with a square-integrable weak derivative. Primes on angular functions denote derivatives in
the indicated angular variable.

A positive continuous $\pi$-periodic function $h$ is a support function
if and only if
\begin{equation}\label{eq:curvature-measure-support}
 h+h''\geq0
\end{equation}
in the sense of distributions. For $h=h_K$, the resulting nonnegative
measure is the planar surface area measure $S_K$, viewed in angular
coordinates. Thus, for every smooth periodic test function $\zeta$,
\[
 \int_{\mathbb S^1}\zeta\,\dd S_K
 =\int_0^{2\pi}h_K(\zeta+\zeta'')\,\dd\theta.
\]
In particular, $S_K(\mathbb S^1)=\int_0^{2\pi}h_K\,\dd\theta$, and
\begin{equation}\label{eq:surface-area-linear}
 S_{(1-t)K+tL}=(1-t)S_K+tS_L,\qquad 0\leq t\leq1.
\end{equation}
Uniform convergence of support functions therefore implies weak
convergence of their surface area measures. This is the measure
continuity used in the equality argument below.

A body is of class $C^\infty_+$ if its support function is smooth and
$h+h''>0$; in this case $\dd S_K=(h+h'')\,\dd\theta$.
For an arbitrary $K\in\mathcal K_e^2$, choose a nonnegative smooth even
approximate identity $\eta_\varepsilon$ on $\T_\pi$ and set
\begin{equation}\label{eq:smoothing}
 \begin{aligned}
 h_\varepsilon&=h_K*\eta_\varepsilon+\varepsilon,\\
 h_\varepsilon+h_\varepsilon''
 &=(h_K+h_K'')*\eta_\varepsilon+\varepsilon>0.
 \end{aligned}
\end{equation}
These functions define origin-symmetric $C^\infty_+$ bodies converging
to $K$ in the Hausdorff metric. Smoothing a fixed pair of bodies in
this way preserves common positive inradius and finite circumradius
bounds throughout their Minkowski segments.

\begin{remark}\label{rem:lower-dimensional}
The inequalities in Theorem~\ref{thm:sharp-threshold}, without the
equality classification, also hold for lower-dimensional
origin-symmetric compact convex sets in the stated parameter range.
For $0<\lambda<1$, apply the full-dimensional result to
$K+\varepsilon B^2$ and $L+\varepsilon B^2$. As $\varepsilon\downarrow0$,
the defining Wulff data decrease to those of the limiting combination,
and the corresponding compact Wulff sets decrease to that combination.
The weighted integrals converge by dominated convergence, since
$|x|^{q-2}$ is locally integrable. The cases $\lambda=0,1$ are identities.
\end{remark}

\vskip20pt
\section{The logarithmic case}\label{sec:p0}

We record the planar logarithmic theorem and its equality cases from
\cite{BLYZ,XY}, which provide the starting point for the range $p>0$.

\begin{theorem}\label{thm:p0-case}
	Let \(0<q\leq2\), \(K,L\in\mathcal K_e^2\), and \(0\leq\lambda\leq1\). Then
	\begin{equation}\label{eq:p0-log-bm}
		\Vq\bigl((1-\lambda)\mathbin{\cdot}K+_0
		\lambda\mathbin{\cdot}L\bigr)
		\geq \Vq(K)^{1-\lambda}\Vq(L)^\lambda.
	\end{equation}
	For \(0<\lambda<1\), equality holds if and only if
	\(K,L\) are positive dilates, or \(q=2\) and they are
	parallelograms with corresponding sides parallel.
	For every \(q>2\), the inequality fails at \(\lambda=1/2\)
	for some unconditional rectangles.
\end{theorem}

\begin{proof}
The cases \(\lambda=0,1\) are identities, so assume
\(0<\lambda<1\).
Suppose first that \(0<q<2\).
The planar logarithmic Brunn--Minkowski inequality of
Xiong and Yang \cite[Theorem~1.3]{XY} gives
\eqref{eq:p0-log-bm}, with equality precisely for positive
dilates. Their normalization satisfies
\[
\widehat V_q(K)=\frac12\int_{\mathbb S^1}\rho_K^q
=\frac q2\Vq(K),
\]
and the fixed factor \(q/2\) cancels from both sides,
preserving both the inequality and its equality condition.

When \(q=2\), we have \(\Vq(A)=|A|\).
The inequality and its stated equality cases therefore
follow from the planar log-Brunn--Minkowski theorem of
B\"or\"oczky, Lutwak, Yang, and Zhang
\cite[Theorem~1.3]{BLYZ}.

Finally, for every \(q>2\), the construction in
\cite[Theorem~1.4 and Step~1 of its proof]{XY}
provides two unconditional rectangles for which the
midpoint inequality is strictly reversed.
The same normalization identity transfers this strict
inequality to \(\Vq\). Hence no \(q>2\) satisfies
\eqref{eq:p0-log-bm} for all origin-symmetric convex bodies.
\end{proof}

\vskip20pt
\section{\texorpdfstring{The case $0<p<1$}{The case 0 < p < 1}}\label{sec:positive-p}

We first establish the inequality in the range $0<q\leq2$; the argument also covers the case $p=1$.

\begin{proposition}\label{prop:positive-low}
	Let \(0<p\leq1\), \(0<q\leq2\), \(K,L\in\mathcal K_e^2\) and \(0\leq\lambda\leq1\). Then
	\[
	\Vq\bigl((1-\lambda)\mathbin{\cdot}K+_p
	\lambda\mathbin{\cdot}L\bigr)^{p/q}
	\geq
	(1-\lambda)\Vq(K)^{p/q}
	+\lambda\Vq(L)^{p/q}.
	\]
	If \(0<\lambda<1\), equality holds if and only if
	\(L=cK\) for some \(c>0\).
\end{proposition}

\begin{proof}
	The cases \(\lambda=0,1\) are immediate. Assume \(0<\lambda<1\), and set
	\[
	a=\Vq(K)^{1/q},\qquad b=\Vq(L)^{1/q},\qquad
	K_0=a^{-1}K,\qquad L_0=b^{-1}L.
	\]
	Then \(\Vq(K_0)=\Vq(L_0)=1\). Define \(C>0\) and \(\theta\in(0,1)\) by
	\[
	C^p=(1-\lambda)a^p+\lambda b^p,\qquad
	\theta=\frac{\lambda b^p}{C^p}.
	\]
	By the definition of \(L_p\) addition,
	\begin{equation}\label{eq:positive-normalization}
		(1-\lambda)\mathbin{\cdot}K+_p\lambda\mathbin{\cdot}L
		=C\bigl((1-\theta)\mathbin{\cdot}K_0+_p
		\theta\mathbin{\cdot}L_0\bigr).
	\end{equation}
	The weighted power-mean inequality gives
	\[
	(1-\theta)\mathbin{\cdot}K_0+_p\theta\mathbin{\cdot}L_0
	\supset G_\theta
	:=\Wulff{h_{K_0}^{1-\theta}h_{L_0}^{\theta}}.
	\]
	Hence, by monotonicity and Theorem~\ref{thm:p0-case},
	\[
	\Vq\bigl((1-\theta)\mathbin{\cdot}K_0+_p
	\theta\mathbin{\cdot}L_0\bigr)
	\geq\Vq(G_\theta)
	\geq\Vq(K_0)^{1-\theta}\Vq(L_0)^\theta=1.
	\]
	Together with \eqref{eq:positive-normalization} and homogeneity,
	this proves the desired inequality.

	Suppose equality holds. Then both inequalities above are
	equalities. If \(0<q<2\), Theorem~\ref{thm:p0-case} implies
	that \(K_0,L_0\) are positive dilates, and their normalization
	gives \(K_0=L_0\).

	Now let \(q=2\) and suppose \(K_0\neq L_0\).
	By the equality classification in \cite{BLYZ}, there are
	nonparallel origin-centered segments \(I_1,I_2\) and
	constants \(c_1,c_2>0\) such that
	\[
	K_0=I_1+I_2,\qquad L_0=c_1I_1+c_2I_2.
	\]
	Since \(|K_0|=|L_0|=1\), we have \(c_1c_2=1\).
	Set
	\[
	d_i=\bigl((1-\theta)+\theta c_i^p\bigr)^{1/p},
	\qquad i=1,2.
	\]
	For \(0<p\leq1\), the weighted \(p\)-mean is concave and
	positively homogeneous, hence superadditive. Thus, for every
	\(u\in\mathbb S^1\),
	\[
	\begin{aligned}
		\bigl((1-\theta)h_{K_0}(u)^p
		+\theta h_{L_0}(u)^p\bigr)^{1/p}
		\geq d_1h_{I_1}(u)+d_2h_{I_2}(u)
		=h_{d_1I_1+d_2I_2}(u).
	\end{aligned}
	\]
	It follows that
	\[
	(1-\theta)\mathbin{\cdot}K_0+_p\theta\mathbin{\cdot}L_0
	\supset d_1I_1+d_2I_2.
	\]
	Moreover, \(K_0\neq L_0\) implies \(c_1\neq1\), so
	\[
	(d_1d_2)^p
	=1+\theta(1-\theta)\bigl(c_1^p+c_1^{-p}-2\bigr)>1.
	\]
	Since \(\Vq\) is area when \(q=2\), we obtain
	\[
	\Vq\bigl((1-\theta)\mathbin{\cdot}K_0+_p
	\theta\mathbin{\cdot}L_0\bigr)
	\geq |d_1I_1+d_2I_2|
	=d_1d_2 |I_1+I_2|>1,
	\]
	contradicting equality. Therefore \(K_0=L_0\) also when \(q=2\).

	Since \(K_0=L_0\), the definitions of \(K_0\) and \(L_0\) give \(L=(b/a)K\).
	Conversely, positive dilates give equality directly by homogeneity.
\end{proof}

Next, we will construct a counterexample when \(q>2\). Fix \(0<p<1\) and \(q>2\). For \(\gamma\in(0,\pi)\), set
\[
a_\gamma=\bigl(\cos(\gamma/2),\sin(\gamma/2)\bigr),
\qquad
b_\gamma=\bigl(\cos(\gamma/2),-\sin(\gamma/2)\bigr),
\]
and let
\[
S_a=[-a_\gamma,a_\gamma],\qquad
S_b=[-b_\gamma,b_\gamma],\qquad
P_\gamma=S_a+S_b.
\]
For \(t\in(0,1)\), define
\[
\begin{aligned}
	K_{t,\gamma}=(1+t)S_a+(1-t)S_b, \qquad
	L_{t,\gamma}=(1-t)S_a+(1+t)S_b,
\end{aligned}
\]
and
\[
c_p(t)=
\big(\frac{(1+t)^p+(1-t)^p}{2}\big)^{1/p}.
\]

Reflection across the first coordinate axis interchanges
\(K_{t,\gamma}\) and \(L_{t,\gamma}\), so
\begin{equation}\label{eq:congruent-moments}
	\Vq(K_{t,\gamma})=\Vq(L_{t,\gamma}).
\end{equation}

\begin{theorem}\label{thm:counterexample-q-greater-than-two}
	For every \(t\in(0,1)\) and \(\gamma\in(0,\pi)\),
	\begin{equation}\label{eq:exact-midpoint}
		\frac12\mathbin{\cdot}K_{t,\gamma}
		+_p\frac12\mathbin{\cdot}L_{t,\gamma}
		=c_p(t)P_\gamma.
	\end{equation}
	For some such \(t,\gamma\),
	\begin{equation}\label{eq:strict-counterexample}
		\Vq\bigl(c_p(t)P_\gamma\bigr)
		<\Vq(K_{t,\gamma})
		=\Vq(L_{t,\gamma}).
	\end{equation}
	Thus, for every \(0<p<1\) and \(q>2\), the \(L_p\)
	Brunn--Minkowski inequality fails at the midpoint for some
	origin-symmetric parallelograms.
\end{theorem}

\begin{proof}
	\textbf{Step 1.}
	The support functions satisfy
	\[
		h_{K_{t,\gamma}}=(1+t)h_{S_a}+(1-t)h_{S_b}, \quad \text{and} \quad
		h_{L_{t,\gamma}}=(1-t)h_{S_a}+(1+t)h_{S_b}.
	\]
	For \(0<p<1\), the \(p\)-mean is concave and positively
	homogeneous, hence superadditive. Therefore, pointwise on
	\(\mathbb S^1\),
	\[
	\big(
	\frac{h_{K_{t,\gamma}}^p+h_{L_{t,\gamma}}^p}{2}
	\big)^{1/p}
	\geq c_p(t)(h_{S_a}+h_{S_b})
	=h_{c_p(t)P_\gamma}.
	\]
	By the definition of the Wulff combination,
	\[
	c_p(t)P_\gamma
	\subseteq
	\frac12\mathbin{\cdot}K_{t,\gamma}
	+_p\frac12\mathbin{\cdot}L_{t,\gamma}.
	\]
	The preceding support-function inequality is an equality
	whenever \(u\perp a_\gamma\) or \(u\perp b_\gamma\), since
	one of the segment support functions then vanishes.
	In these four directions, the defining bounds of the Wulff
	combination coincide with the four supporting lines of
	\(c_p(t)P_\gamma\). Their half-spaces intersect in precisely
	this parallelogram, giving the reverse inclusion.
	This proves \eqref{eq:exact-midpoint}.

	\textbf{Step 2.}
	Recall that
	\begin{align*}
		\Vq(C)=\int_C |z|^{q-2}\,\dd z.
	\end{align*}
	Under the change of variables \(z=xa_\gamma+yb_\gamma\),
	the absolute Jacobian is \(\sin\gamma\), and
	\begin{align*}
		|xa_\gamma+yb_\gamma|^2
		=(x+y)^2\cos^2(\gamma/2)
		+(x-y)^2\sin^2(\gamma/2).
	\end{align*}
	Fix \(t\in(0,1)\). The corresponding coordinate domains for
	\(K_{t,\gamma}\) and \(P_\gamma\) are
	\[
	[-(1+t),1+t]\times[-(1-t),1-t]
	\quad\text{and}\quad [-1,1]^2.
	\]
	On both rectangles, \(|xa_\gamma+yb_\gamma|\leq2\).
	Since \(q>2\), dominated convergence, after cancellation of
	the common Jacobian, gives
	\begin{align*}
		\lim_{\gamma\to0}
		\frac{\Vq(K_{t,\gamma})}{\Vq(P_\gamma)}
		=
		\frac{
			\displaystyle
			\int_{-(1+t)}^{1+t}\int_{-(1-t)}^{1-t}
			|x+y|^{q-2}\,\dd y\,\dd x
		}{
			\displaystyle
			\int_{-1}^{1}\int_{-1}^{1}
			|x+y|^{q-2}\,\dd y\,\dd x
		}.
	\end{align*}
	Direct integration yields
	\begin{align*}
		\int_{-(1+t)}^{1+t}\int_{-(1-t)}^{1-t}
		|x+y|^{q-2}\,\dd y\,\dd x
		=
		\frac{2^{q+1}}{q(q-1)}(1-t^q),
		\qquad 0\leq t<1.
	\end{align*}
	Applying this formula at \(t\) and at \(0\), we obtain
	\begin{equation}\label{eq:simple-limit}
		\lim_{\gamma\to0}
		\frac{\Vq(K_{t,\gamma})}{\Vq(P_\gamma)}
		=1-t^q.
	\end{equation}

	On the other hand, Taylor expansion at \(t=0\) gives
	\begin{align*}
		c_p(t)^q
		=1-\frac{q(1-p)}2t^2+O(t^4).
	\end{align*}
	Since \(q>2\),
	\[
	(1-t^q)-c_p(t)^q
	=\frac{q(1-p)}2t^2+o(t^2)>0
	\]
	for all sufficiently small \(t>0\).
	Fix such a \(t\), so that
	\begin{align*}
		1-t^q>c_p(t)^q.
	\end{align*}
	By \eqref{eq:simple-limit}, we may then choose
	\(\gamma>0\) sufficiently small that
	\[
	\frac{\Vq(K_{t,\gamma})}{\Vq(P_\gamma)}
	>c_p(t)^q.
	\]
	Using \eqref{eq:exact-midpoint}, homogeneity, and
	\eqref{eq:congruent-moments}, we conclude that
	\[
	\begin{aligned}
		\Vq\bigl(
		\tfrac12\mathbin{\cdot}K_{t,\gamma}
		+_p\tfrac12\mathbin{\cdot}L_{t,\gamma}
		\bigr)
		=c_p(t)^q\Vq(P_\gamma)
		<\Vq(K_{t,\gamma})
		=\Vq(L_{t,\gamma}).
	\end{aligned}
	\]
	This proves \eqref{eq:strict-counterexample}.
	Since \(p/q>0\), it also gives the claimed failure of the
	midpoint \(L_p\) Brunn--Minkowski inequality.
\end{proof}

\begin{remark}\label{rem:not-unconditional}
For the chosen parameters, \(K_{t,\gamma}\) and \(L_{t,\gamma}\)
are non-unconditional parallelograms exchanged by a reflection,
while their \(L_p\) midpoint is a rhombus.
Smooth approximation of their support functions, preserving this
reflection relation, also yields \(C^\infty_+\) counterexamples:
the strict deficit persists by continuity of the \(L_p\)-Wulff
operation and \(\Vq\).

These examples are consistent with Hu--Ivaki
\cite[Theorem~1.4]{HI}, who proved the inequality for unconditional
bodies when \(0<p<1\), \(q=2+\alpha\), and
\[
0<\alpha\leq\frac{2p(1-p)}{2-p}.
\]
Their \(L_0\) result \cite[Theorem~1.3]{HI} covers \(0<q<2\),
and is likewise consistent with the unconditional counterexamples
for \(q>2\) in Theorem~\ref{thm:p0-case}.
\end{remark}

\vskip20pt
\section{\texorpdfstring{The Minkowski case: $2<q<4$}{The Minkowski case: 2 < q < 4}}\label{sec:p1}

Throughout this section, let \(2<q<4\). We first derive the local
inequality, then prove the required spectral estimates, and finally
establish the global inequality and its equality case.

Let \(K\in\mathcal K_e^2\) be of class \(C^\infty_+\), with support
function \(h=h_K\). In outer-normal coordinates, write
\[
R=h+h'',\quad r=\sqrt{h^2+(h')^2},\quad
x(\theta)=h(\theta)e_\theta+h'(\theta)e_\theta^\perp,
\]
where \(e_\theta=(\cos\theta,\sin\theta)\) and
\(e_\theta^\perp=(-\sin\theta,\cos\theta)\).
Then \(x'=Re_\theta^\perp\), \(|x|=r\), and the divergence theorem gives
\begin{equation}\label{eq:volume-support}
	\Vq(K)=\frac1q\int_0^{2\pi}hRr^{q-2}\,\dd\theta.
\end{equation}
We write \(\Vq(h)\) for the functional evaluated at the convex body with support
function \(h\). Primes in the calculations below denote angular derivatives,
with the angular variable specified when coordinates change.

To study the concavity of \(\Vq^{1/q}\) along Minkowski segments, we first derive the following variation formulas.

\begin{lemma}\label{lem:second-variation}
Let \(\psi\) be a smooth \(\pi\)-periodic function. Then
\[
\Big.\frac{\dd}{\dd t}\Vq(h+t\psi)\Big|_{t=0}
=\int_0^{2\pi}\psi Rr^{q-2}\,\dd\theta.
\]
In particular, for \(\psi=fh\),
\[
\Big.\frac{\dd}{\dd t}\Vq(h+t f h)\Big|_{t=0}
=\int_0^{2\pi}fhRr^{q-2}\,\dd\theta,
\]
and
\[
\Big.\frac{\dd^2}{\dd t^2}\Vq(h+t f h)\Big|_{t=0}
=(q-1)\int_0^{2\pi}f^2hRr^{q-2}\,\dd\theta
-\int_0^{2\pi}(f')^2h^2r^{q-2}\,\dd\theta.
\]
\end{lemma}

\begin{proof}
For the variation $h_t=h+t\psi$, we have
\[
R_t=h_t+h_t''
=
R+t(\psi+\psi'').
\]
Here the two primes in $\psi''$ indicate the second derivative with
respect to the angular variable $\theta$.

Furthermore,
\[
r_t=(h_t^2+(h_t')^2)^{1/2}.
\]
Differentiating with respect to the variation parameter \(t\)
and evaluating at \(t=0\), we obtain
\[
\frac{\dd}{\dd t}r_t\Big|_{t=0}
=
\frac{h_t\psi+h_t'\psi'}{r_t}\Big|_{t=0}=\frac{h\psi+h'\psi'}{r}.
\]

Differentiating the integral formula for $\widetilde V_q$ with respect to
$t$ and then setting $t=0$, we get
\[
\begin{aligned}
\Big.\frac{\dd}{\dd t}\Vq(h+t\psi)\Big|_{t=0}
=
\frac1q\int_0^{2\pi}
\psi Rr^{q-2}
+h(\psi+\psi'')r^{q-2}
+(q-2)hRr^{q-4}(h\psi+h'\psi')
\dd\theta .
\end{aligned}
\]
Integrating the term involving $\psi''$ by parts, we obtain
\begin{align*}
\int_0^{2\pi}h\psi''r^{q-2}\,\dd\theta&=-\int_0^{2\pi}
(h'r^{q-2}+h(r^{q-2})')\psi'\,\dd\theta\\
&=-\int_0^{2\pi}
\left(
h'r^{q-2}
+
(q-2)hh'Rr^{q-4}
\right)\psi'\,\dd\theta\\
&=\int_0^{2\pi}
(h'r^{q-2})'\psi\,\dd\theta-\int_0^{2\pi}(q-2)hh'Rr^{q-4}
\psi'\,\dd\theta\\
&=\int_0^{2\pi}(h''r^{q-2}
+(q-2)(h')^2Rr^{q-4})\psi\,\dd\theta-\int_0^{2\pi}(q-2)hh'Rr^{q-4}
\psi'\,\dd\theta
\end{align*}

Using
$
r^2=h^2+(h')^2$, and 
$R=h+h'',
$
we obtain
\begin{align}\label{eq:first-variation}
\Big.\frac{\dd}{\dd t}\Vq(h+t\psi)\Big|_{t=0}
=
\int_0^{2\pi}
\psi Rr^{q-2}\,\dd\theta .
\end{align}
In particular, choosing
$
\psi=fh,
$
gives
\[
\Big.\frac{\dd}{\dd t}\Vq(h+t f h)\Big|_{t=0}
=
\int_0^{2\pi}
fhRr^{q-2}\,\dd\theta .
\]

Differentiating \eqref{eq:first-variation} once more and integrating
the term containing \(\psi''\) by parts gives
\[
\Big.\frac{\dd^2}{\dd t^2}\Vq(h+t f h)\Big|_{t=0}
=\int_0^{2\pi}(r^{q-2}+(q-2)hRr^{q-4})\psi^2\,\dd\theta
-\int_0^{2\pi}(\psi')^2r^{q-2}\,\dd\theta.
\]
Substitute \(\psi=fh\) and use
\[
\int_0^{2\pi}2ff'hh'r^{q-2}\,\dd\theta
=-\int_0^{2\pi}f^2(hh'r^{q-2})'\,\dd\theta.
\]
The coefficient of \(f^2\) then reduces to
\((q-1)hRr^{q-2}\), which proves the desired result.
\end{proof}

Define the weighted mean
\[
\bar f=
\frac{\int_0^{2\pi}fhRr^{q-2}\,\dd\theta}
{\int_0^{2\pi}hRr^{q-2}\,\dd\theta}.
\]
For \(h_t=h+tfh\), the chain rule and Lemma \ref{lem:second-variation} give
\begin{equation}\label{eq:p1-second-deficit}
	\begin{aligned}
		\frac{\dd^2}{\dd t^2}\Vq(h_t)^{1/q}\Big|_{t=0}
		={}&-\frac1q\Vq(K)^{1/q-1}
		\Big(
		\int_0^{2\pi}(f')^2h^2r^{q-2}\,\dd\theta\\
		&\qquad
		-(q-1)\int_0^{2\pi}(f-\bar f)^2hRr^{q-2}\,\dd\theta
		\Big).
	\end{aligned}
\end{equation}
Thus local concavity along Minkowski segments is equivalent to
\begin{equation}\label{eq:support-poincare}
	\int_0^{2\pi}(f')^2h^2r^{q-2}\,\dd\theta
	\geq
	(q-1)\int_0^{2\pi}(f-\bar f)^2hRr^{q-2}\,\dd\theta.
\end{equation}

We next rewrite \eqref{eq:support-poincare} in radial coordinates. Let \(\phi\) be the polar angle of \(x(\theta)\), and let
\(v=h_{K^\circ}\). Then
\begin{equation}\label{eq:polar-coordinate-identities}
	v(\phi)=\frac1{r(\theta)},\quad
	\frac{\dd\phi}{\dd\theta}=\frac{hR}{r^2},\quad
	v+v''=\frac{r^3}{h^3R}>0.
\end{equation}
Here derivatives of \(v\) are with respect to \(\phi\).
The second identity follows from
\(\dd\phi/\dd\theta=\det(x,x')/|x|^2\); differentiating
\(v=1/r\) gives \(v'=-h'/(hr)\) and then the third identity.
For \(g(\phi(\theta))=f(\theta)\), we obtain
\begin{equation}\label{eq:polar-integral-identities}
	v^{-q}\,\dd\phi=hRr^{q-2}\,\dd\theta,\quad
	\frac{g'^2v^{1-q}}{v+v''}\,\dd\phi
	=(f')^2h^2r^{q-2}\,\dd\theta.
\end{equation}
All functions are \(\pi\)-periodic. On the projective circle
\(\T_\pi=\mathbb R/\pi\mathbb Z\), put
\[
\bar g=\frac{\int_{\T_\pi}gv^{-q}\,\dd\phi}
{\int_{\T_\pi}v^{-q}\,\dd\phi}.
\]
The change of variables gives \(\bar g=\bar f\), and
\eqref{eq:support-poincare} is equivalent to
\begin{equation}\label{eq:polar-poincare}
	\int_{\T_\pi}\frac{g'^2v^{1-q}}{v+v''}\,\dd\phi
	\geq(q-1)\int_{\T_\pi}(g-\bar g)^2v^{-q}\,\dd\phi.
\end{equation}

We prove the required spectral estimates next, using
\(\alpha=q-2\in(0,2)\).
We begin with the harmonic-weight model; the notation below will also be used in the subsequent comparison argument.
For an interval \(J=[a,b]\) and a nonnegative weight \(w\), set
\[
 \|z\|_{H^1(w)}^2=\int_Jw(z^2+z'^2).
\]
Here \(H^1(w)\) denotes the completion of
\(C^\infty([a,b])\) in this weighted Sobolev norm.
Let \(\mathcal D_D(J;w)\) be the closure in this norm of the smooth functions satisfying \(z(a)=z(b)=0\), and let \(\mathcal D_{ND}(J;w)\) be the corresponding closure when only \(z(b)=0\) is imposed.
Define
\[
 \lambda_D(J;w)
 =\inf_{\substack{0\ne z\in\mathcal D_D(J;w)}}
  \frac{\int_Jwz'^2}{\int_Jwz^2},
\quad
 \lambda_{ND}(J;w)
 =\inf_{\substack{0\ne z\in\mathcal D_{ND}(J;w)}}
  \frac{\int_Jwz'^2}{\int_Jwz^2}.
\]
Thus \(ND\) means a natural boundary condition at the left endpoint and a Dirichlet condition at the right endpoint.
This definition also covers the degeneracy of \(w=\sin^\alpha t\) at \(t=0\).

\begin{lemma}\label{lem:harmonic-model}
	Let \(0<\alpha<2\).
	\begin{enumerate}[label=\textup{(\roman*)}]
		\item If \(I=[a,b]\) has length at most \(\pi/2\), and
		\(s\in C^2(I)\) satisfies
		\[
		s>0\quad\hbox{on }I,\quad s''+s=0\quad\hbox{on }(a,b),
		\]
		then
		\[
		\lambda_D(I;s^\alpha)>\alpha+1.
		\]
		\item If \(0\leq a_0<b_0\leq\pi/2\), then
		\[
		\lambda_{ND}([a_0,b_0];\sin^\alpha t)\geq\alpha+1.
		\]
		The inequality is strict when \(a_0>0\), and
		\(\lambda_{ND}([0,\pi/2];\sin^\alpha t)=\alpha+1\).
	\end{enumerate}
\end{lemma}

\begin{proof}
\textbf{Step 1.} We first prove \textup{(ii)}.
For a smooth \(z\) on \([0,\pi/2]\) with \(z(\pi/2)=0\),
direct differentiation gives, on \((0,\pi/2)\),
\begin{equation}\label{eq:harmonic-ground-state-pointwise}
	\sin^\alpha t\,(z'^2-(\alpha+1)z^2)
	=\sin^\alpha t\cos^2t
	\big(\big(\frac{z}{\cos t}\big)'\big)^2
	-(\sin^\alpha t\tan t\,z^2)'.
\end{equation}
The boundary term tends to zero at both endpoints. Indeed, as 
\[
\sin^\alpha t\tan t\,z(t)^2=O(t^{\alpha+1}), \quad t\downarrow0
\]
whereas \(z(t)=O(\pi/2-t)\) gives
\[
\sin^\alpha t\tan t\,z(t)^2=O(\pi/2-t), \quad t\uparrow\pi/2.
\]
Integrating \eqref{eq:harmonic-ground-state-pointwise} over
\([\varepsilon,\pi/2-\varepsilon]\) and letting
\(\varepsilon\downarrow0\), we obtain
\begin{equation}\label{eq:harmonic-ground-state-identity}
	\int_0^{\pi/2}\sin^\alpha t\,(z'^2-(\alpha+1)z^2)\,\dd t
	=\int_0^{\pi/2}\sin^\alpha t\cos^2t
	\big(\big(\frac{z}{\cos t}\big)'\big)^2\,\dd t
	\geq0.
\end{equation}
By density, the inequality extends to
\(\mathcal D_{ND}([0,\pi/2];\sin^\alpha t)\).
Equality is attained by \(z(t)=\cos t\), so
\begin{equation}\label{eq:harmonic-full-mixed}
	\lambda_{ND}([0,\pi/2];\sin^\alpha t)=\alpha+1.
\end{equation}

Now extend a smooth test function on \([a_0,b_0]\) constantly
over \([0,a_0]\) and by zero over \([b_0,\pi/2]\). The extension
belongs to the mixed form domain on \([0,\pi/2]\), since it is
piecewise smooth and continuous and can be approximated in
\(H^1\) by smooth functions vanishing at \(\pi/2\).
Its Dirichlet energy is unchanged, while its weighted squared
\(L^2\)-norm can only increase. Thus
\[
\lambda_{ND}([a_0,b_0];\sin^\alpha t)
\geq\lambda_{ND}([0,\pi/2];\sin^\alpha t)=\alpha+1,
\]
where density gives the assertion on the full form domain.
If \(a_0>0\), let \(z\) be a positive first mixed eigenfunction.
The Neumann condition \(z'(a_0)=0\) and uniqueness for the
initial-value problem imply \(z(a_0)>0\).
The same extension therefore preserves its Dirichlet energy
and strictly increases its weighted \(L^2\)-norm, giving
\[
\lambda_{ND}([a_0,b_0];\sin^\alpha t)>\alpha+1.
\]

\textbf{Step 2.} We next prove \textup{(i)}. A translation and cancellation of a positive constant factor
in the Rayleigh quotient reduce the problem to
\(s(x)=\cos x\), with \(I\subset(-\pi/2,\pi/2)\).
Set \(\lambda_0=\alpha+1\). A fundamental system for
\begin{equation}\label{eq:harmonic-critical-equation}
	-(\cos^\alpha x\,y')'=\lambda_0\cos^\alpha x\,y
\end{equation}
is
\[
y_1(x)=\sin x,\quad
y_2(x)=\alpha\sin x\,J_\alpha(x)-\cos^{1-\alpha}x,
\]
where $J_\alpha(x)=\int_0^x\cos^{-\alpha}t\,\dd t$.
Direct differentiation verifies the equation; linear
independence follows from \(y_1(0)=0\), \(y_1'(0)=1\), and
\(y_2(0)=-1\). For \(x\ne0\), put
\[
R(x)=\frac{y_2(x)}{y_1(x)}
=\alpha J_\alpha(x)-\frac{\cos^{1-\alpha}x}{\sin x}.
\]
Then
\begin{equation}\label{eq:Rprime}
	R'(x)=\frac{\cos^{-\alpha}x}{\sin^2x}>0,
	\quad R(-x)=-R(x).
\end{equation}
A nonzero multiple of \(y_1\) has only one zero in
\((-\pi/2,\pi/2)\). Any other solution is nonzero at the
origin, and its zeros satisfy \(R(x)=\text{constant}\).
By \eqref{eq:Rprime}, it has at most one zero on each side
of the origin. Thus, if it has two zeros, write them as
\(-\xi\) and \(\eta\), where \(\xi,\eta>0\); they satisfy
\[
R(\xi)+R(\eta)=0.
\]

For \(0<x<\pi/2\), set \(S(x)=R(x)+R(\pi/2-x)\).
Since
\[
\frac{R'(x)}{R'(\pi/2-x)}=(\cot x)^{2-\alpha},
\]
the assumption \(\alpha<2\) shows that \(S\) increases on
\((0,\pi/4)\) and decreases on \((\pi/4,\pi/2)\).
Moreover, integrating
\[
(2\sin t\cos^{1-\alpha}t)'
=\alpha\cos^{-\alpha}t
+(2-\alpha)\cos^{-\alpha}t\cos(2t)
\]
over \([0,\pi/4]\) gives
\begin{align*}
	2^{\alpha/2}-\alpha\int_0^{\pi/4}\cos^{-\alpha}t\,\dd t
	=(2-\alpha)\int_0^{\pi/4}\cos^{-\alpha}t\cos(2t)\,\dd t>0.
\end{align*}
Consequently,
\[
S(x)\leq S(\pi/4)=2R(\pi/4)<0.
\]
If \(\xi+\eta\leq\pi/2\), monotonicity of \(R\) would give
\[
0=R(\xi)+R(\eta)
\leq R(\xi)+R(\pi/2-\xi)=S(\xi)<0,
\]
a contradiction. Hence two distinct zeros of a nonzero
solution of \eqref{eq:harmonic-critical-equation} are
separated by more than \(\pi/2\).

Finally, let \(u>0\) on \((a,b)\) be a first Dirichlet
eigenfunction, and write \(\nu=\lambda_D(I;\cos^\alpha x)\).
Suppose \(\nu\leq\lambda_0\), and let \(y\) solve
\eqref{eq:harmonic-critical-equation} with
\[
y(a)=0,\qquad y'(a)=u'(a)>0.
\]
The zero separation just proved implies \(y>0\) on \((a,b]\).
The two differential equations yield
\[
\bigl(\cos^\alpha x\,(u'y-uy')\bigr)'
=(\lambda_0-\nu)\cos^\alpha x\,uy\geq0.
\]
Integrating and using \(u(a)=u(b)=y(a)=0\), we obtain
\[
\cos^\alpha b\,u'(b)y(b)
=(\lambda_0-\nu)\int_a^b\cos^\alpha x\,uy\,\dd x\geq0.
\]
This contradicts \(u'(b)<0\) and \(y(b)>0\).
Here the endpoint derivative signs follow from positivity
of \(u\) and uniqueness for the regular initial-value problem.
Therefore \(\nu>\lambda_0=\alpha+1\). 
\end{proof}

The following lemma extends the Dirichlet eigenvalue estimate from harmonic weights to general support-function weights.
\begin{lemma}
	\label{lem:dirichlet-support}
	Let \(0<\alpha<2\), let \(I=(a,b)\) be an interval of length at most
	\(\pi/2\), and let \(v\in C([a,b])\) be such that
	\[
	v>0 \quad \text{on } [a,b],
	\qquad
	v''+v\geq 0 \quad \text{on } I
	\]
	in the sense of distributions. Then
	\[
	\lambda_D(I;v^\alpha)>\alpha+1.
	\]
\end{lemma}

\begin{proof}

	\textbf{Step 1.} 
	Set
	$
	\lambda=\lambda_D(I;v^\alpha),
	$
	and let \(z\) be a corresponding first Dirichlet eigenfunction,
	normalized so that
	\begin{equation}\label{eq:strict-z-eigen}
		-(v^\alpha z')'=\lambda v^\alpha z,
		\quad
		z>0\quad\hbox{on }(a,b),
		\quad
		z(a)=z(b)=0.
	\end{equation}
	Since \(v\) is continuous and strictly positive, one has
	\(z\in C^1([a,b])\) and \(v^\alpha z'\in C^1([a,b])\).
	Moreover,
	\[
	(v^\alpha z')'=-\lambda v^\alpha z<0
	\quad\hbox{on }(a,b).
	\]
	Thus \(v^\alpha z'\) is strictly decreasing. Consequently, \(z\)
	has a unique maximum point \(m\in(a,b)\), and
	\begin{align*}
		z'>0\quad\hbox{on }(a,m),
		\quad
		z'(m)=0,
		\quad
		z'<0\quad\hbox{on }(m,b).
	\end{align*}
	Let
	$
	\mu=v''+v.
	$
	The distribution \(\mu\) is a nonnegative Radon measure. In
	particular, \(v'\) has one-sided limits at \(m\), and
	$
	v'_+(m)-v'_-(m)=\mu(\{m\})\geq0.
	$
	Choose
	\[
	d\in[v'_-(m),v'_+(m)]
	\]
	and define
	$
	s(x)=v(m)\cos(x-m)+d\sin(x-m).
	$
	Then
	\[
	s''+s=0,
	\quad
	s(m)=v(m),
	\quad
	s'(m)=d.
	\]
	The measure-valued variation-of-constants formula gives, for
	\(x\geq m\),
	\begin{align*}
		v(x)-s(x)
		=
		\bigl(v'_+(m)-d\bigr)\sin(x-m)+
		\int_{(m,x)}\sin(x-t)\,\dd\mu(t),
	\end{align*}
	whereas, for \(x\leq m\),
	\begin{align*}
		v(x)-s(x)
		=
		\bigl(d-v'_-(m)\bigr)\sin(m-x)+
		\int_{(x,m)}\sin(t-x)\,\dd\mu(t).
	\end{align*}
	All the terms on the right-hand sides are nonnegative, because
	\(b-a\leq\pi/2\), \(\mu\geq0\), and
	\(v'_-(m)\leq d\leq v'_+(m)\). Hence
	$
		v\geq s
$
on $[a,b]$.	

	\textbf{Step 2.} 
	On the component of \(\{s>0\}\cap I\) containing \(m\), set
	\[
	W=sv'-s'v.
	\]
	Differentiation in the sense of measures gives
	\begin{align*}
		\dd W=s\,\dd\mu,
	\end{align*}
	while
	\[
	W(m-)=v(m)\bigl(v'_-(m)-d\bigr)\leq0,\qquad
	W(m+)=v(m)\bigl(v'_+(m)-d\bigr)\geq0.
	\]
	Since \(s>0\) and \(\mu\geq0\), it follows that \(W\leq0\)
	to the left of \(m\) and \(W\geq0\) to its right.
	The identity
	\[
	\Big(\frac{v}{s}\Big)'=\frac{W}{s^2}
	\]
	holds almost everywhere.
	It follows that
	$
	Q=(s/v)^\alpha
	$
	is nondecreasing to the left of \(m\) and nonincreasing to its
	right. Set \(A_0=v^\alpha zz'\). By \eqref{eq:strict-z-eigen},
	\begin{align*}
		A_0'=v^\alpha(z'^2-\lambda z^2).
	\end{align*}
	Moreover, \(A_0\geq0\) to the left of \(m\) and \(A_0\leq0\)
	to its right. Thus \(A_0\,\dd Q\) is a nonnegative measure
	on the region under consideration.

	Suppose first that \(s>0\) on \([a,b]\). Since
	\(A_0(a)=A_0(b)=0\), integration by parts gives
		\begin{align*}
			\int_a^b s^\alpha(z'^2-\lambda z^2)\,\dd x
			=\int_a^b Q A_0'\,\dd x
			=-\int_a^b A_0\,\dd Q\leq0.
		\end{align*}
	In particular, this calculation does not require derivatives of $v$ at $a, b$. The Rayleigh characterization
	and Lemma ~\ref{lem:harmonic-model} \,\textup{(i)} imply 
	\[
	\lambda
	\geq
	\frac{\int_a^b s^\alpha z'^2\,\dd x}
	{\int_a^b s^\alpha z^2\,\dd x}
	\geq
	\lambda_D(I;s^\alpha)> \alpha+1 .
	\]

	\textbf{Step 3.} The mixed comparison when \(s\) has a zero.
	The zeros of a nonzero solution of \(s''+s=0\) are separated
	by \(\pi\). Since \(b-a\leq\pi/2\), there is at most one zero
	\(c\) in \([a,b]\), and \(c\ne m\).
	If \(c<m\), then
	\[
	s(x)=C\sin(x-c)>0\quad\hbox{on }[m,b]
	\]
	for some \(C>0\), and
	\[
	0<m-c<b-c\leq b-a\leq\frac{\pi}{2}.
	\]
	On \([m,b]\), \(Q\) is continuous and monotone, and
	\(A_0(m)=A_0(b)=0\). Hence the same integration by parts yields
	\[
	\int_m^b s^\alpha(z'^2-\lambda z^2)\,\dd x
	=-\int_m^b A_0\,\dd Q\leq0.
	\]
	Under \(t=x-c\), the restriction of \(z\) is an admissible
	mixed test function on \([m-c,b-c]\), and the factor
	\(C^\alpha\) cancels from its Rayleigh quotient. Thus
	Lemma~\ref{lem:harmonic-model}\,\textup{(ii)} gives
	\[
	\lambda\geq
	\frac{\int_m^b s^\alpha z'^2\,\dd x}
	{\int_m^b s^\alpha z^2\,\dd x}
	\geq\lambda_{ND}\bigl([m-c,b-c];\sin^\alpha t\bigr)
	>\alpha+1,
	\]
	where strictness follows from \(m-c>0\).
	If \(c>m\), apply the same argument on \([a,m]\) with
	\(t=c-x\); the resulting mixed interval is
	\([c-m,c-a]\), whose left endpoint is positive and whose
	right endpoint is at most \(\pi/2\).
	This proves the lemma.
\end{proof}

Here \(H^1(\T_\pi)\) is the space of \(\pi\)-periodic functions
whose values and weak first derivatives are square integrable.
For a positive continuous \(\pi\)-periodic support function \(v\),
let \(\lambda_1(v^{-\alpha})\) denote the first nonzero periodic
eigenvalue of
\[
-(v^{-\alpha}Y')'=\lambda v^{-\alpha}Y.
\]
For \(Y\in H^1(\T_\pi)\), denote its \(v^{-\alpha}\)-weighted mean by
\[
\bar Y=
\frac{\int_{\T_\pi}v^{-\alpha}Y\,\dd\phi}
{\int_{\T_\pi}v^{-\alpha}\,\dd\phi}.
\]
The interval comparison now yields the periodic spectral gap needed below.

\Needspace{6\baselineskip}
\begin{proposition}
\label{lem:uniform-strict-gap}
	Let \(0<\alpha<2\).
	\begin{enumerate}[label=\textup{(\roman*)}]
		\item For every positive continuous \(\pi\)-periodic support
		function \(v\), one has
		\[
		\lambda_1(v^{-\alpha})>\alpha+1.
		\]
		Consequently, every \(Y\in H^1(\T_\pi)\) satisfies
		\begin{equation}\label{eq:key-gap}
			\int_{\T_\pi}v^{-\alpha}Y'^2\,\dd\phi
			\geq(\alpha+1)\int_{\T_\pi}v^{-\alpha}(Y-\bar Y)^2\,\dd\phi.
		\end{equation}
		Equality in \eqref{eq:key-gap} holds if and only if \(Y\) is constant.
		\item Given \(0<m_0<M_0<\infty\), there exists
		\(\delta=\delta(\alpha,m_0,M_0)>0\) such that every such
		\(v\) with \(m_0\leq v\leq M_0\) satisfies
		\begin{equation}\label{eq:uniform-key-gap}
			\int_{\T_\pi}v^{-\alpha}Y'^2\,\dd\phi
			\geq(\alpha+1+\delta)
			\int_{\T_\pi}v^{-\alpha}(Y-\bar Y)^2\,\dd\phi
		\end{equation}
		for every \(Y\in H^1(\T_\pi)\).
	\end{enumerate}
\end{proposition}

\begin{proof}
	Let
	$
	\lambda_1=\lambda_1(v^{-\alpha}),
	$
	and let \(Y\) be a corresponding first nonzero eigenfunction.
	By the variational characterization of \(\lambda_1\),
	\[
	-\bigl(v^{-\alpha}Y'\bigr)'
	=\lambda_1v^{-\alpha}Y,
	\quad
	\int_{\T_\pi}v^{-\alpha}Y\,\dd\phi=0,
	\]
	in the weak sense. Set
	$
	z=v^{-\alpha}Y'.
	$
	The eigenvalue equation gives
	\[
	-z'=\lambda_1v^{-\alpha}Y,
	\quad
	-v^\alpha z'=\lambda_1Y.
	\]
	Differentiating the second identity and using \(Y'=v^\alpha z\),
	we obtain
	\begin{equation}\label{eq:lemma58-z}
		-\bigl(v^\alpha z'\bigr)'=\lambda_1v^\alpha z
	\end{equation}
	weakly on \(\T_\pi\). The identity
	\(z'=-\lambda_1v^{-\alpha}Y\) implies that \(z\) is continuous.

	By the periodicity of \(Y\),
	\[
	\int_{\T_\pi}v^\alpha z\,\dd\phi
	=
	\int_{\T_\pi}Y'\,\dd\phi
	=0.
	\]
	Moreover, \(z\not\equiv0\), since otherwise \(Y\) would be
	constant. Hence \(z\) changes sign. Let \(I=(a,b)\) be a shortest
	nodal interval of \(z\). Since the circle has length \(\pi\) and
	\(z\) has at least two nodal intervals,
	\[
	|I|\leq\frac{\pi}{2}.
	\]
	On \(I\), \(z\) has one sign and vanishes at the endpoints.
	Therefore \(z\) is a first Dirichlet eigenfunction of
	\eqref{eq:lemma58-z} on \(I\), and consequently
	\[
	\lambda_1(v^{-\alpha})
	=
	\lambda_D(I;v^\alpha).
	\]
	Lemma~\ref{lem:dirichlet-support} now yields
	\[
	\lambda_1(v^{-\alpha})
	=
	\lambda_D(I;v^\alpha)
	>
	\alpha+1.
	\]

	For an arbitrary nonconstant \(Y\in H^1(\T_\pi)\), subtracting
	\(\bar Y\) and using the variational characterization gives
	\[
	\int_{\T_\pi}v^{-\alpha}Y'^2\,\dd\phi
	\geq\lambda_1(v^{-\alpha})
	\int_{\T_\pi}v^{-\alpha}(Y-\bar Y)^2\,\dd\phi
	>(\alpha+1)\int_{\T_\pi}v^{-\alpha}(Y-\bar Y)^2\,\dd\phi.
	\]
	For constant \(Y\), both sides of \eqref{eq:key-gap} vanish.
	This proves \textup{(i)}.

	For \textup{(ii)}, the support functions satisfying
	\[
	m_0\leq v\leq M_0,\quad v+v''\geq0
	\]
	form a compact subset of \(C^0(\T_\pi)\).
	Indeed, they are the support functions of convex bodies that contain \(m_0B^2\) and are contained in \(M_0B^2\), and hence are uniformly bounded and uniformly Lipschitz.
	Moreover, if \(v_j\to v\) uniformly, then \(v_j^{-\alpha}\to v^{-\alpha}\) uniformly, and the Rayleigh-Ritz characterization of the first nonzero eigenvalue gives
	\[
	\lambda_1(v_j^{-\alpha})\longrightarrow
	\lambda_1(v^{-\alpha}).
	\]
	Compactness together with \textup{(i)} therefore yields a strictly positive uniform gap \(\delta\).
\end{proof}

For a positive \(\pi\)-periodic function \(v\) and a \(\pi\)-periodic function \(g\), write
\[
\bar g=\frac{\int_{\T_\pi}gv^{-(\alpha+2)}\,\dd\phi}
{\int_{\T_\pi}v^{-(\alpha+2)}\,\dd\phi}
\]
whenever the integrals are finite. We next convert the periodic spectral gap into the polar estimate required above.

\begin{lemma}\label{lem:hardy-cauchy}
	Let \(0<\alpha<2\), and let \(v>0\) be a smooth
	\(\pi\)-periodic support function with \(v+v''>0\).
	Then every smooth \(\pi\)-periodic function \(g\) satisfies
	\begin{equation}\label{eq:polar-alpha}
		\int_{\T_\pi}\frac{g'^2v^{-(\alpha+1)}}{v+v''}\,\dd\phi
		\geq(\alpha+1)\int_{\T_\pi}(g-\bar g)^2v^{-(\alpha+2)}\,\dd\phi.
	\end{equation}
\end{lemma}

\begin{proof}
	\textbf{Step 1.} An estimate for a periodic primitive.
	Subtracting \(\bar g\) leaves the derivative unchanged, so assume
	\(\int_{\T_\pi}gv^{-(\alpha+2)}\,\dd\phi=0\).
	Choose the additive constant of a periodic primitive so that
	\begin{equation}\label{eq:normalized-primitive}
		P'=gv^{-(\alpha+2)},\quad \int_{\T_\pi}Pv\,\dd\phi=0.
	\end{equation}
	This is possible because the derivative has zero integral and \(\int v>0\).
	Set \(Y=Pv^{\alpha+1}\). Then
	\[
	\int_{\T_\pi}v^{-\alpha}Y\,\dd\phi=0,\quad
	\int_{\T_\pi}v^{-\alpha}Y^2\,\dd\phi
	=\int_{\T_\pi}P^2v^{\alpha+2}\,\dd\phi.
	\]
	Expanding \(Y'\) and integrating the mixed term by parts gives
	\[
	\begin{aligned}
		\int_{\T_\pi}v^{-\alpha}Y'^2\,\dd\phi
		={}&\int_{\T_\pi}P'^2v^{\alpha+2}\,\dd\phi
		+(\alpha+1)\int_{\T_\pi}(P^2)'v^{\alpha+1}v'\,\dd\phi\\
		&+(\alpha+1)^2\int_{\T_\pi}P^2v^\alpha(v')^2\,\dd\phi\\
		={}&\int_{\T_\pi}P'^2v^{\alpha+2}\,\dd\phi
		-(\alpha+1)\int_{\T_\pi}P^2v^{\alpha+1}v''\,\dd\phi.
	\end{aligned}
	\]
	Consequently,
	\begin{equation}\label{eq:hardy-identity}
		\begin{aligned}
			&\int_{\T_\pi}v^{-\alpha}Y'^2\,\dd\phi
			-(\alpha+1)\int_{\T_\pi}v^{-\alpha}Y^2\,\dd\phi\\
			&\quad=
			\int_{\T_\pi}P'^2v^{\alpha+2}\,\dd\phi
			-(\alpha+1)\int_{\T_\pi}P^2(v+v'')v^{\alpha+1}\,\dd\phi.
		\end{aligned}
	\end{equation}
	Proposition~\ref{lem:uniform-strict-gap}\,\textup{(i)} applied to \(Y\) therefore yields
	\begin{align*}
		\int_{\T_\pi}P'^2v^{\alpha+2}\,\dd\phi
		\geq(\alpha+1)\int_{\T_\pi}P^2(v+v'')v^{\alpha+1}\,\dd\phi.
	\end{align*}

	\textbf{Step 2.} The Cauchy-Schwarz estimate.
	By \eqref{eq:normalized-primitive} and periodic integration by parts,
	\[
	\int_{\T_\pi}g^2v^{-(\alpha+2)}\,\dd\phi
	=\int_{\T_\pi}P'^2v^{\alpha+2}\,\dd\phi
	=-\int_{\T_\pi}g'P\,\dd\phi.
	\]
	Hence
		\begin{align*}
			\Big(\int_{\T_\pi}g^2v^{-(\alpha+2)}\,\dd\phi\Big)^2
			\leq{}&\Big(\int_{\T_\pi}
			\frac{g'^2v^{-(\alpha+1)}}{v+v''}\,\dd\phi\Big)
			\Big(\int_{\T_\pi}P^2(v+v'')v^{\alpha+1}\,\dd\phi\Big)\\
			\leq{}&\frac1{\alpha+1}
			\Big(\int_{\T_\pi}\frac{g'^2v^{-(\alpha+1)}}{v+v''}\,\dd\phi\Big)
			\int_{\T_\pi}g^2v^{-(\alpha+2)}\,\dd\phi.
		\end{align*}
	For \(g\not\equiv0\), division by the last integral proves
	\eqref{eq:polar-alpha}; the zero function satisfies it as well.
	Restoring the weighted mean completes the proof.
\end{proof}

Taking \(\alpha=q-2\) proves \eqref{eq:polar-poincare}, and hence
\eqref{eq:support-poincare}. To retain rigidity under smooth
approximation, we strengthen this estimate uniformly.

\begin{proposition}
	\label{prop:quantitative-rigidity}
	Fix \(2<q<4\) and \(0<m_0<M_0<\infty\).
	Let \(v\) be a smooth \(\pi\)-periodic support function satisfying
	\[
	m_0\leq v\leq M_0,
	\qquad
	v+v''>0,
	\]
	and let \(g\) be a smooth \(\pi\)-periodic function such that
	\[
	\int_{\T_\pi} g v^{-q}\,\dd\phi=0.
	\]
	Then there exists a constant
	$
	\kappa=\kappa(q,m_0,M_0)>0,
	$
	independent of \(v\) and \(g\), such that
	\begin{equation}\label{eq:quantitative-polar-gap}
		\int_{\T_\pi}\frac{g'^2v^{1-q}}{v+v''}\,\dd\phi
		-(q-1)\int_{\T_\pi}g^2v^{-q}\,\dd\phi
		\geq
		\kappa\int_{\T_\pi}g^2v^{-q}\,\dd\phi.
	\end{equation}
\end{proposition}

\begin{proof}
	\textbf{Step 1.} 
	The mean-zero condition on \(g\) allows us to choose a smooth
	\(\pi\)-periodic \(P\) satisfying
	\begin{equation}\label{eq:quantitative-primitive}
		P'=gv^{-q},\quad \int_{\T_\pi}Pv\,\dd\phi=0.
	\end{equation}
	Set \(Y=Pv^{q-1}\). Since \(\alpha=q-2\),
	\[
	\int_{\T_\pi}v^{-\alpha}Y\,\dd\phi=0,\quad
	\int_{\T_\pi}v^{-\alpha}Y^2\,\dd\phi
	=\int_{\T_\pi}P^2v^q\,\dd\phi,
	\]
	and
	\begin{align*}
		\int_{\T_\pi}P'^2v^q\,\dd\phi
		=\int_{\T_\pi}g^2v^{-q}\,\dd\phi.
	\end{align*}
	Proposition~\ref{lem:uniform-strict-gap}\,\textup{(ii)} and
	\eqref{eq:hardy-identity} therefore give a constant
	\(\delta=\delta(q,m_0,M_0)>0\), independent of \(v\), such that
	\begin{equation}\label{eq:A-H-uniform}
		\begin{aligned}
			\int_{\T_\pi}g^2v^{-q}\,\dd\phi
			-(q-1)\int_{\T_\pi}P^2(v+v'')v^{q-1}\,\dd\phi
			\geq
			\delta\int_{\T_\pi}P^2v^q\,\dd\phi.
		\end{aligned}
	\end{equation}
	On the other hand, integration by parts gives
	\[
	\int_{\T_\pi}g^2v^{-q}\,\dd\phi
	=\int_{\T_\pi}gP'\,\dd\phi
	=-\int_{\T_\pi}g'P\,\dd\phi.
	\]
	Cauchy-Schwarz then yields
	\begin{equation}\label{eq:A-E-H}
		\Big(\int_{\T_\pi}g^2v^{-q}\,\dd\phi\Big)^2
		\leq
		\Big(\int_{\T_\pi}
		\frac{g'^2v^{1-q}}{v+v''}\,\dd\phi\Big)
		\Big(\int_{\T_\pi}P^2(v+v'')v^{q-1}\,\dd\phi\Big).
	\end{equation}
	In particular, these two estimates imply
	\[
	\int_{\T_\pi}\frac{g'^2v^{1-q}}{v+v''}\,\dd\phi
	\geq(q-1)\int_{\T_\pi}g^2v^{-q}\,\dd\phi.
	\]

	\textbf{Step 2.} The uniform positive gap.
	Suppose that no uniform \(\kappa>0\) exists.
	By homogeneity and the preceding inequality, we may choose
	admissible \(v_j,g_j\) such that
	\[
	\int_{\T_\pi}g_jv_j^{-q}\,\dd\phi=0,\quad
	\int_{\T_\pi}g_j^2v_j^{-q}\,\dd\phi=1,
	\]
	and
	\begin{equation}\label{eq:quantitative-energy-sequence}
		\int_{\T_\pi}\frac{g_j'^2v_j^{1-q}}{v_j+v_j''}\,\dd\phi
		\longrightarrow q-1.
	\end{equation}
	Choose \(P_j\) as in \eqref{eq:quantitative-primitive}.
	Then \(\int_{\T_\pi}P_j'^2v_j^q\,\dd\phi=1\), and
	\eqref{eq:A-H-uniform}-\eqref{eq:A-E-H} imply
	\[
	\frac{1}{
		\displaystyle\int_{\T_\pi}
		\frac{g_j'^2v_j^{1-q}}{v_j+v_j''}\,\dd\phi}
	\leq
	\int_{\T_\pi}P_j^2(v_j+v_j'')v_j^{q-1}\,\dd\phi
	\leq\frac1{q-1}.
	\]
	Consequently,
	\begin{equation}\label{eq:quantitative-positive-limit}
		\int_{\T_\pi}P_j^2(v_j+v_j'')v_j^{q-1}\,\dd\phi
		\longrightarrow\frac1{q-1}.
	\end{equation}
	Moreover, \eqref{eq:A-H-uniform} gives
	\[
	0\leq\delta\int_{\T_\pi}P_j^2v_j^q\,\dd\phi
	\leq
	1-(q-1)\int_{\T_\pi}P_j^2(v_j+v_j'')v_j^{q-1}\,\dd\phi
	\longrightarrow0.
	\]
	Since \(v_j\geq m_0\), it follows that
\[
			\int_{\T_\pi}P_j^2\,\dd\phi
			\leq m_0^{-q}\int_{\T_\pi}P_j^2v_j^q\,\dd\phi
			\longrightarrow0,
\]
and
\[
			\int_{\T_\pi}P_j'^2\,\dd\phi
			\leq m_0^{-q}\int_{\T_\pi}P_j'^2v_j^q\,\dd\phi
			=m_0^{-q}.
\]
	The normalization \(\int_{\T_\pi}P_jv_j\,\dd\phi=0\)
	ensures that \(P_j\) has a zero. Integrating
	\((P_j^2)'=2P_jP_j'\) from such a point and applying the
	Cauchy-Schwarz, we obtain
	\[
	\|P_j\|_\infty^2
	\leq
	2\Big(\int_{\T_\pi}P_j^2\,\dd\phi\Big)^{1/2}
	\Big(\int_{\T_\pi}P_j'^2\,\dd\phi\Big)^{1/2}
	\longrightarrow0.
	\]
	Using periodicity, \(v_j+v_j''>0\), and \(v_j\leq M_0\),
	we conclude that
	\[
	\begin{aligned}
		0
		&\leq\int_{\T_\pi}P_j^2(v_j+v_j'')v_j^{q-1}\,\dd\phi
		\leq
		M_0^{q-1}\|P_j\|_\infty^2
		\int_{\T_\pi}(v_j+v_j'')\,\dd\phi\\
		&=
		M_0^{q-1}\|P_j\|_\infty^2
		\int_{\T_\pi}v_j\,\dd\phi
		\leq\pi M_0^q\|P_j\|_\infty^2\longrightarrow0.
	\end{aligned}
	\]
	This contradicts \eqref{eq:quantitative-positive-limit}
	and proves \eqref{eq:quantitative-polar-gap}.
\end{proof}

\begin{remark}\label{rem:quantitative-support-gap}
	Proposition~\ref{prop:quantitative-rigidity} has the following
	equivalent form in outer-normal coordinates.
	Let \(K\) be an origin-symmetric convex body with
	\(v=h_{K^\circ}\) satisfying the assumptions of the proposition,
	and write
	\[
	h=h_K,\quad R=h+h'',\quad r=\sqrt{h^2+(h')^2}.
	\]
	The bounds on \(v\) are equivalent to
	\(M_0^{-1}\leq r\leq m_0^{-1}\).
	For a smooth \(\pi\)-periodic \(f\), define the weighted mean
	\[
	\bar f=
	\frac{\int_0^{2\pi}fhRr^{q-2}\,\dd\theta}
	{\int_0^{2\pi}hRr^{q-2}\,\dd\theta}.
	\]
	Then, with the same constant \(\kappa\),
	\begin{equation}\label{eq:quantitative-support-gap}
		\begin{aligned}
			\int_0^{2\pi}(f')^2h^2r^{q-2}\,\dd\theta
			-(q-1)\int_0^{2\pi}(f-\bar f)^2hRr^{q-2}\,\dd\theta
			\geq
			\kappa\int_0^{2\pi}(f-\bar f)^2hRr^{q-2}\,\dd\theta.
		\end{aligned}
	\end{equation}

	Indeed, if \(\theta\) is the outer-normal angle and \(\phi\)
	is the polar angle of the corresponding boundary point,
	the change of variables satisfies
	\[
	v(\phi)=\frac1{r(\theta)},\quad
	\frac{\dd\phi}{\dd\theta}=\frac{hR}{r^2},\quad
	v+v''=\frac{r^3}{h^3R}.
	\]
	For \(g(\phi(\theta))=f(\theta)\), this gives
	\[
	v^{-q}\,\dd\phi=hRr^{q-2}\,\dd\theta,\quad
	\frac{g'^2v^{1-q}}{v+v''}\,\dd\phi
	=(f')^2h^2r^{q-2}\,\dd\theta,
	\]
	where the primes on \(v,g\) refer to \(\phi\), and those on
	\(h,f\) refer to \(\theta\).
	Since all functions are \(\pi\)-periodic,
	\[
	\frac{\int_{\T_\pi}gv^{-q}\,\dd\phi}
	{\int_{\T_\pi}v^{-q}\,\dd\phi}
	=\bar f.
	\]
	Applying \eqref{eq:quantitative-polar-gap} to \(g-\bar f\),
	we obtain
	\[
	\int_{\T_\pi}\frac{g'^2v^{1-q}}{v+v''}\,\dd\phi
	\geq
	(q-1+\kappa)\int_{\T_\pi}(g-\bar f)^2v^{-q}\,\dd\phi.
	\]
	Changing variables and multiplying by \(2\) yields
	\[
	\int_0^{2\pi}(f')^2h^2r^{q-2}\,\dd\theta
	\geq
	(q-1+\kappa)\int_0^{2\pi}(f-\bar f)^2hRr^{q-2}\,\dd\theta,
	\]
	which is equivalent to \eqref{eq:quantitative-support-gap}.
\end{remark}

With the local concavity and uniform rigidity estimates in hand, we now pass to arbitrary origin-symmetric convex bodies.
\begin{theorem}
	\label{thm:minkowski-interior}\label{thm:interior-equality}
	Let \(2<q<4\), \(K,L\in\mathcal K_e^2\), and \(0\leq\lambda\leq1\).
	Then
	\begin{equation}\label{eq:minkowski-interior}
		\Vq((1-\lambda)K+\lambda L)^{1/q}
		\geq
		(1-\lambda)\Vq(K)^{1/q}
		+\lambda\Vq(L)^{1/q}.
	\end{equation}
	If \(0<\lambda<1\), equality holds if and only if
	\(L=cK\) for some \(c>0\).
\end{theorem}
\begin{proof}
	\textbf{Step 1.}
	Suppose first that \(K,L\) are of class \(C^\infty_+\), and set
	\[
	K_t=(1-t)K+tL,\qquad
	h_t=(1-t)h_K+th_L,
	\]
	\[
	G(t)=\Vq(K_t)^{1/q},
	\qquad 0\leq t\leq1.
	\]
	Since \(h_t+h_t''>0\), each \(K_t\) is also of class
	\(C^\infty_+\). Write
	\[
	R_t=h_t+h_t'',\qquad
	r_t=\sqrt{h_t^2+(h_t')^2},\qquad
	f_t=\frac{h_L-h_K}{h_t},
	\]
	and let
	\[
	\bar f_t=
	\frac{\int_0^{2\pi}f_t h_tR_t r_t^{q-2}\,\dd\theta}
	{\int_0^{2\pi}h_tR_t r_t^{q-2}\,\dd\theta}.
	\]
	All these functions of \(\theta\) are \(\pi\)-periodic.
	The second-variation formula
	\eqref{eq:p1-second-deficit} gives
	\begin{equation}\label{eq:interior-second-variation}
		\begin{aligned}
			-G''(t)
			=\frac1q\Vq(K_t)^{1/q-1}\Big(
			\int_0^{2\pi}(f_t')^2h_t^2r_t^{q-2}\,\dd\theta
			-(q-1)\int_0^{2\pi}
			(f_t-\bar f_t)^2h_tR_t r_t^{q-2}\,\dd\theta
			\Big).
		\end{aligned}
	\end{equation}
	Here primes on \(h_t,f_t\) denote derivatives in \(\theta\),
	while \(G''\) denotes the second derivative in \(t\).

	Proposition~\ref{lem:uniform-strict-gap} gives
	\eqref{eq:key-gap}. Lemma~\ref{lem:hardy-cauchy} then yields
	\eqref{eq:polar-poincare}, and the radial change of variables
	gives \eqref{eq:support-poincare}.
	Applying \eqref{eq:support-poincare} to \(K_t\) and \(f_t\)
	shows that the right-hand side of
	\eqref{eq:interior-second-variation} is nonnegative.
	Hence \(G''(t)\leq0\), so \(G\) is concave. Therefore
	\eqref{eq:minkowski-interior} holds whenever \(K,L\) are of
	class \(C^\infty_+\).

	We now let \(K,L\in\mathcal K_e^2\) be arbitrary, with no
	smoothness assumption. Choose a nonnegative smooth even
	approximate identity \(\eta_j\) on \(\T_\pi\), let
	\(\varepsilon_j\downarrow0\), and define
	\[
	h_{K_j}=h_K*\eta_j+\varepsilon_j,
	\qquad
	h_{L_j}=h_L*\eta_j+\varepsilon_j.
	\]
	Since \(h_K+h_K''\) and \(h_L+h_L''\) are nonnegative
	measures,
	\[
	h_{K_j}+h_{K_j}''
	=(h_K+h_K'')*\eta_j+\varepsilon_j>0,
	\]
	and similarly for \(L_j\). Thus \(h_{K_j}\) and \(h_{L_j}\)
	are the support functions of origin-symmetric
	\(C^\infty_+\) bodies \(K_j,L_j\). Moreover,
	\[
	h_{K_j}\longrightarrow h_K,\qquad
	h_{L_j}\longrightarrow h_L
	\quad\text{uniformly}.
	\]
	Hence \(K_j\to K\) and \(L_j\to L\) in the Hausdorff metric.
	Applying Step~1 to \(K_j,L_j\) and using the continuity of
	\(\Vq\) under Hausdorff convergence, we obtain
	\eqref{eq:minkowski-interior} for all
	\(K,L\in\mathcal K_e^2\).

	\textbf{Step 2.}
	Suppose equality holds in \eqref{eq:minkowski-interior} for
	some \(\lambda_0\in(0,1)\). Set
	\[
	K_t=(1-t)K+tL,\qquad G(t)=\Vq(K_t)^{1/q}.
	\]
	By Step~1, \(G\) is concave. Since it meets its endpoint chord
	at the interior point \(\lambda_0\), \(G\) must be affine on
	\([0,1]\).

	Using the approximating bodies from Step~1, define
	\[
	K_{t,j}=(1-t)K_j+tL_j,\quad
	h_{t,j}=(1-t)h_{K_j}+th_{L_j}, \quad G_j(t)=\Vq(K_{t,j})^{1/q}.
	\]
	Then \(h_{t,j}\to h_t=(1-t)h_K+th_L\) uniformly in
	\((t,\theta)\). Since \(K,L\in\mathcal K_e^2\), there exist
	constants \(0<r_0<R_0<\infty\), independent of \(j,t\), such
	that
	\begin{equation}\label{eq:interior-common-bounds}
		r_0B^2\subset K_j,K_{t,j}\subset R_0B^2,
		\qquad 0\leq t\leq1.
	\end{equation}
	The continuity of \(\Vq\) gives
	\begin{equation}\label{eq:interior-uniform-convergence}
		G_j\longrightarrow G
		\quad\text{uniformly on }[0,1].
	\end{equation}

	For the smooth segments \(K_{t,j}\), write
	\[
	R_{t,j}=h_{t,j}+h_{t,j}'',\quad
	r_{t,j}=\sqrt{h_{t,j}^2+(h_{t,j}')^2}, \quad f_{t,j}=\frac{h_{L_j}-h_{K_j}}{h_{t,j}},
	\]
	and let \(\bar f_{t,j}\) denote the
	\(h_{t,j}R_{t,j}r_{t,j}^{q-2}\)-weighted mean of \(f_{t,j}\),
	as in Step~1.
	The polar support functions satisfy
	\[
	R_0^{-1}\leq h_{K_{t,j}^\circ}\leq r_0^{-1}.
	\]
	Hence Proposition~\ref{prop:quantitative-rigidity}, in the
	form \eqref{eq:quantitative-support-gap}, applies with a
	constant \(\kappa>0\) independent of \(j,t\). Combining it
	with \eqref{eq:interior-second-variation}, we obtain
	\[
	-G_j''(t)\geq
	\frac{\kappa}{q}\Vq(K_{t,j})^{1/q-1}
	\int_0^{2\pi}(f_{t,j}-\bar f_{t,j})^2
	h_{t,j}R_{t,j}r_{t,j}^{q-2}\,\dd\theta.
	\]
	For smooth \(K_{t,j}\),
	\[
	\dd S_{K_{t,j}}=R_{t,j}\,\dd\theta.
	\]
	The bounds \eqref{eq:interior-common-bounds} give
	\(h_{t,j}r_{t,j}^{q-2}\geq r_0^{q-1}\) and a uniform upper
	bound for \(\Vq(K_{t,j})\). Since \(1/q-1<0\), there is a
	constant \(c_0>0\), independent of \(j,t\), such that
	\begin{equation}\label{eq:global-second-deficit}
		-G_j''(t)\geq
		c_0\int_{\mathbb S^1}
		(f_{t,j}-\bar f_{t,j})^2\,\dd S_{K_{t,j}}.
	\end{equation}

	Integration by parts on \([0,\lambda_0]\) and
	\([\lambda_0,1]\) yields
	\begin{equation}\label{eq:green-deficit}
		\begin{aligned}
			&G_j(\lambda_0)
			-(1-\lambda_0)G_j(0)-\lambda_0G_j(1)\\
			&\qquad
			=(1-\lambda_0)\int_0^{\lambda_0}
			t(-G_j''(t))\,\dd t
			+\lambda_0\int_{\lambda_0}^1
			(1-t)(-G_j''(t))\,\dd t.
		\end{aligned}
	\end{equation}
	By \eqref{eq:interior-uniform-convergence} and the affinity of
	\(G\), the left-hand side tends to zero.
	Fix \(0<a<b<1\). On \([a,b]\), the weight on the right-hand
	side of \eqref{eq:green-deficit} is bounded below by
	\[
	\min\{a(1-\lambda_0),\,\lambda_0(1-b)\}>0.
	\]
	Since \(-G_j''\geq0\), it follows from
	\eqref{eq:global-second-deficit} that
	\[
	\int_a^b\int_{\mathbb S^1}
	(f_{t,j}-\bar f_{t,j})^2
	\,\dd S_{K_{t,j}}\,\dd t
	\longrightarrow0.
	\]
	Thus we may choose \(t_j\in[a,b]\) such that
	\begin{equation}\label{eq:interior-vanishing-variance}
		\int_{\mathbb S^1}
		(f_{t_j,j}-\bar f_{t_j,j})^2
		\,\dd S_{K_{t_j,j}}
		\longrightarrow0.
	\end{equation}
	After passing to a subsequence,
	\(t_j\to t_*\in[a,b]\subset(0,1)\). Then
	\[
	K_{t_j,j}\longrightarrow K_{t_*},
	\qquad
	f_{t_j,j}\longrightarrow
	f:=\frac{h_L-h_K}{h_{K_{t_*}}}
	\quad\text{uniformly}.
	\]
	Since
	\[
	|\bar f_{t_j,j}|
	\leq \|f_{t_j,j}\|_\infty,
	\]
	we may pass to a further subsequence such that
	\(\bar f_{t_j,j}\to c\) for some \(c\in\mathbb R\).
	By the weak continuity of surface area measures,
	$
	S_{K_{t_j,j}}\rightharpoonup S_{K_{t_*}}.
	$
	Passing to the limit in
	\eqref{eq:interior-vanishing-variance} gives
	\[
	\int_{\mathbb S^1}(f-c)^2\,\dd S_{K_{t_*}}=0.
	\]
	Since \(f\) is continuous,
	\begin{equation}\label{eq:interior-constant-on-support}
		f=c
		\quad\text{on }\operatorname{supp}S_{K_{t_*}}.
	\end{equation}

	\textbf{Step 3.}
	Write
	\[
	h_*=h_{K_{t_*}},
	\qquad
	w=h_L-h_K-ch_*.
	\]
	By \eqref{eq:interior-constant-on-support},
	\[
	w=0
	\quad\text{on }\operatorname{supp}S_{K_{t_*}}.
	\]
	In the plane, the surface area measure is linear with respect to Minkowski addition, so
	\[
	S_{K_{t_*}}=(1-t_*)S_K+t_*S_L.
	\]
	Since \(0<t_*<1\), if \(I\) is a connected component of
	$
	\mathbb S^1\setminus\operatorname{supp}S_{K_{t_*}},
	$
	then both \(S_K\) and \(S_L\) vanish on \(I\). Hence
	\[
	w''+w=0,\qquad
	h_*''+h_*=0
	\quad\text{on }I
	\]
	in the sense of distributions, and therefore classically.

	The endpoints of \(I\) belong to
	\(\operatorname{supp}S_{K_{t_*}}\), so \(w\) vanishes at both
	endpoints. Moreover, \(h_*>0\) on \(\overline I\).
	A nonzero solution of \(h_*''+h_*=0\) is a translated
	multiple of \(\cos\); hence its strict positivity on
	\(\overline I\) implies \(|I|<\pi\).
	A solution of \(w''+w=0\) vanishing at two points less than
	\(\pi\) apart must be identically zero. Thus \(w=0\) on every
	complementary component and, together with
	\eqref{eq:interior-constant-on-support}, on the whole circle.
	Therefore
	\[
	h_L-h_K
	=c\bigl((1-t_*)h_K+t_*h_L\bigr),
	\]
	or equivalently,
	\[
	(1-ct_*)h_L
	=\bigl(1+c(1-t_*)\bigr)h_K.
	\]
	Neither coefficient can vanish, and since \(h_K,h_L>0\),
	their ratio is positive. Hence \(h_L\) is a positive constant
	multiple of \(h_K\), so
	$
	L=c_0K
	$
	for some \(c_0>0\).

	Conversely, if \(L=c_0K\) with \(c_0>0\), then
	\[
	(1-\lambda)K+\lambda L
	=(1-\lambda+\lambda c_0)K.
	\]
	Using the homogeneity
	\(\Vq(cK)=c^q\Vq(K)\), equality follows in
	\eqref{eq:minkowski-interior}.
\end{proof}

\vskip20pt
\section{The endpoint and sharpness}\label{sec:endpoint}
The strict harmonic comparison used above becomes critical at $q=4$.
We treat this endpoint by a moment-lifting argument and then recall the
rectangle counterexamples for $q>4$.

%\subsection{\texorpdfstring{The endpoint $q=4$}{The endpoint q = 4}}

The following result is the planar endpoint of
\cite[Theorem~1.5]{LiWanZhang}, based on Hadwiger's moment inequality
\cite{Hadwiger1956}. We give the proof directly for origin-symmetric
planar bodies, using disk fibers in a four-dimensional lift.
\begin{theorem}\label{thm:q4-equality}
	Let \(K,L\in\mathcal K_e^2\) and \(0\leq\lambda\leq1\). Then
	\begin{equation}\label{eq:q4-bm}
		\widetilde V_4\bigl((1-\lambda)K+\lambda L\bigr)^{1/4}
		\geq
		(1-\lambda)\widetilde V_4(K)^{1/4}
		+\lambda\widetilde V_4(L)^{1/4}.
	\end{equation}
	If \(0<\lambda<1\), equality holds if and only if \(L=cK\) for
	some constant \(c>0\).
\end{theorem}

\begin{proof}
	\textbf{Step 1.}
	The cases \(\lambda=0,1\) are immediate, so assume
	\(0<\lambda<1\). Write
	\[
	I(A)=\int_A |x|^2\,\dd x=\widetilde V_4(A),
	\qquad
	T_u(A)=\int_A\langle x,u\rangle^2\,\dd x,
	\quad u\in\mathbb S^1.
	\]
	Since \(I(cA)=c^4I(A)\) for \(c>0\), set
	\[
	a=I(K)^{1/4},\qquad b=I(L)^{1/4},\qquad
	K_0=a^{-1}K,\qquad L_0=b^{-1}L.
	\]
	Then \(I(K_0)=I(L_0)=1\). Define
	\[
	s=(1-\lambda)a+\lambda b,
	\qquad
	\tau=\frac{\lambda b}{s}.
	\]
	Then \(0<\tau<1\) and \(s(1-\tau)=(1-\lambda)a\), so for
	\[
	M=(1-\tau)K_0+\tau L_0,
	\]
	we have
	\[
	(1-\lambda)K+\lambda L=sM,
	\qquad
	I\bigl((1-\lambda)K+\lambda L\bigr)^{1/4}
	=sI(M)^{1/4}.
	\]
	Thus it suffices to prove \(I(M)\geq1\) and characterize equality.

	Define the inertia matrix
	\[
	Q_A=\int_A x\otimes x\,\dd x,
	\]
	so that
	\[
	\langle Q_Au,u\rangle=T_u(A),
	\qquad \operatorname{tr}Q_A=I(A).
	\]
	The symmetric matrix \(D=Q_{K_0}-Q_{L_0}\) has trace zero.
	If \(D\ne0\), choose orthonormal eigenvectors \(e_1,e_2\)
	with eigenvalues \(\gamma,-\gamma\), and set
	\[
	u_1=\frac{e_1+e_2}{\sqrt2},
	\qquad
	u_2=\frac{e_1-e_2}{\sqrt2}.
	\]
	Then \(u_1,u_2\) are orthonormal and
	\(\langle Du_i,u_i\rangle=(\gamma-\gamma)/2=0\).
	If \(D=0\), any orthonormal basis has this property. Hence
	\begin{equation}\label{eq:matched-directional-moments}
		T_{u_i}(K_0)=T_{u_i}(L_0)=:m_i>0,
		\qquad i=1,2.
	\end{equation}
	The moments are positive because the bodies have interior, and
	orthonormality gives
	\[
	|x|^2=\langle x,u_1\rangle^2+\langle x,u_2\rangle^2,
	\qquad m_1+m_2=I(K_0)=1.
	\]

	Fix \(u\in\mathbb S^1\), and write
	\[
	H_u^+=\{x\in\R^2:\langle x,u\rangle\geq0\}.
	\]
	For a compact convex set \(E\subset H_u^+\), define
	\[
	\mathcal C_u(E)
	=\bigl\{(x,z)\in\R^2\times\R^2:
	x\in E,\ |z|\leq\langle x,u\rangle\bigr\}.
	\]
	This set is compact. Its convexity follows from that of \(E\)
	and the estimate, for \((x_i,z_i)\in\mathcal C_u(E)\) and
	\(0\leq t\leq1\),
	\[
	\begin{aligned}
		|(1-t)z_1+tz_2|
		\leq(1-t)|z_1|+t|z_2|
		\leq\langle(1-t)x_1+tx_2,u\rangle.
	\end{aligned}
	\]

	For \(A\in\{K_0,L_0,M\}\), we have \(0\in\operatorname{int}A\).
	Thus \((\varepsilon u,0)\) is an interior point of
	\(\mathcal C_u(A\cap H_u^+)\) for sufficiently small
	\(\varepsilon>0\). These lifts are therefore four-dimensional
	convex bodies with nonempty interior.

	The disk fiber over \(x\in E\) has area
	\(\pi\langle x,u\rangle^2\), so Fubini's theorem gives
	\begin{align*}
		|\mathcal C_u(E)|_4
		=\pi\int_E\langle x,u\rangle^2\,\dd x,
	\end{align*}
	where \(|\cdot|_4\) denotes four-dimensional volume.
	Since \(\langle x,u\rangle^2\) is even, origin symmetry gives
	\begin{equation}\label{eq:q4-half-lift-volume}
		|\mathcal C_u(A\cap H_u^+)|_4=\frac{\pi}{2}T_u(A).
	\end{equation}

	The same triangle inequality, now with base points in
	\(K_0\cap H_u^+\) and \(L_0\cap H_u^+\), yields
	\begin{equation}\label{eq:q4-lift-inclusion}
		(1-\tau)\mathcal C_u(K_0\cap H_u^+)
		+\tau\mathcal C_u(L_0\cap H_u^+)
		\subset\mathcal C_u(M\cap H_u^+).
	\end{equation}

	Combining \eqref{eq:q4-half-lift-volume},
	\eqref{eq:q4-lift-inclusion}, and the four-dimensional
	Brunn--Minkowski inequality, we obtain
	\begin{align*}
		\bigl(\tfrac{\pi}{2}T_u(M)\bigr)^{1/4}
		&=|\mathcal C_u(M\cap H_u^+)|_4^{1/4}\\
		&\geq
		\bigl|(1-\tau)\mathcal C_u(K_0\cap H_u^+)
		+\tau\mathcal C_u(L_0\cap H_u^+)\bigr|_4^{1/4}\\
		&\geq
		(1-\tau)|\mathcal C_u(K_0\cap H_u^+)|_4^{1/4}
		+\tau|\mathcal C_u(L_0\cap H_u^+)|_4^{1/4}\\
		&=\bigl(\tfrac{\pi}{2}\bigr)^{1/4}
		\bigl[(1-\tau)T_u(K_0)^{1/4}
		+\tau T_u(L_0)^{1/4}\bigr].
	\end{align*}
	Thus
	\begin{equation}\label{eq:directional-moment-bm}
		T_u(M)^{1/4}
		\geq(1-\tau)T_u(K_0)^{1/4}
		+\tau T_u(L_0)^{1/4}.
	\end{equation}
	Taking \(u=u_i\) and using \eqref{eq:matched-directional-moments}
	gives \(T_{u_i}(M)\geq m_i\). Therefore
	\[
	I(M)=T_{u_1}(M)+T_{u_2}(M)
	\geq m_1+m_2=1.
	\]
	By the normalization above, this proves \eqref{eq:q4-bm}.

	\textbf{Step 2.}
	Suppose equality holds, so \(I(M)=1\). Since
	\[
	T_{u_i}(M)\geq m_i,\qquad i=1,2,
	\]
	and \(T_{u_1}(M)+T_{u_2}(M)=1=m_1+m_2\), we have
	\[
	T_{u_i}(M)=T_{u_i}(K_0)=T_{u_i}(L_0)=m_i,
	\qquad i=1,2.
	\]

	Fix \(u=u_1\). Then equality holds in the Brunn--Minkowski step
	leading to \eqref{eq:directional-moment-bm}. Hence, by the equality
	condition in the Brunn--Minkowski inequality,
	\[
	\mathcal C_u(L_0\cap H_u^+)
	=r\,\mathcal C_u(K_0\cap H_u^+)+w
	\]
	for some \(r>0\) and \(w\in\R^4\). Since the two lifts have the same
	positive volume, \(r=1\).

	Projecting onto the first two coordinates gives
	\begin{equation}\label{eq:q4-translated-halves}
		L_0\cap H_u^+=(K_0\cap H_u^+)+y
	\end{equation}
	for some \(y\in\R^2\). Since both half-bodies meet \(u^\perp\),
	\eqref{eq:q4-translated-halves} implies \(\langle y,u\rangle=0\).
	Thus
	\[
	L_0\cap u^\perp=(K_0\cap u^\perp)+y.
	\]
	Both sections are centered line segments, so \(y=0\). Hence
	\[
	L_0\cap H_u^+=K_0\cap H_u^+.
	\]
	By origin symmetry, \(L_0=K_0\), and therefore
	\[
	L=\frac ba K.
	\]

	Conversely, if \(L=cK\) with \(c>0\), then
	\[
	(1-\lambda)K+\lambda L=(1-\lambda+\lambda c)K,
	\]
	and homogeneity gives equality in \eqref{eq:q4-bm}.
\end{proof}

%\subsection{\texorpdfstring{Sharpness for $q>4$}{Sharpness for q > 4}}

The rectangle construction of \cite[Section~4.1]{LiWanZhang} shows that
the endpoint is sharp, even within the unconditional class. We include
the calculation in the present normalization.

\begin{proposition}\label{prop:q-greater-4}
	For every \(q>4\), there exist
	\(K,L\in\mathcal K_e^2\) such that
	\[
	\Vq\Big(\frac{K+L}{2}\Big)^{1/q}
	<
	\frac12\Vq(K)^{1/q}
	+\frac12\Vq(L)^{1/q}.
	\]
	In particular, \eqref{eq:main-bm} with \(p=1\) fails for every \(q>4\).
\end{proposition}

\begin{proof}
	Let
	\[
	Q=[-1,1]^2.
	\]
	For \(|t|<1\), consider the rectangles
	\[
	K_t=[-(1+t),1+t]\times[-(1-t),1-t],
	\]
	and
	\[
	L_t=[-(1-t),1-t]\times[-(1+t),1+t].
	\]
	They are obtained from each other by interchanging the coordinate axes, and
	\[
	\frac{K_t+L_t}{2}=Q,
	\qquad
	\Vq(K_t)=\Vq(L_t).
	\]

	We claim that, for every \(q>4\) and all sufficiently small \(t>0\),
	\begin{equation}\label{eq:rectangle-reverse}
		\Vq(K_t)>\Vq(Q).
	\end{equation}
	Indeed, set
	\[
	\Phi(t)=\Vq(K_t)
	=
	\int_{-(1+t)}^{1+t}
	\int_{-(1-t)}^{1-t}
	(x^2+y^2)^{(q-2)/2}\,\dd y\,\dd x.
	\]
	By symmetry, \(\Phi'(0)=0\). Differentiating with respect to the
	integration limits and then integrating by parts gives
	\begin{align*}
		\Phi''(0)
		&=
		8\Big(
		\frac{q-2}{2}
		\int_{-1}^{1}(1+y^2)^{(q-4)/2}\,\dd y
		-2^{(q-2)/2}
		\Big)\\
		&=
		4(q-4)\int_{-1}^{1}
		(1-y^2)(1+y^2)^{(q-6)/2}\,\dd y
		>0.
	\end{align*}
	Hence \(t=0\) is a strict local minimum of \(\Phi\), and
	\eqref{eq:rectangle-reverse} follows for all sufficiently small
	nonzero \(t\).

	On the other hand, \eqref{eq:main-bm} with \(p=1\) and
	\(\lambda=\frac12\) would give
	\[
	\Vq(Q)^{1/q}
	\geq
	\frac12\Vq(K_t)^{1/q}
	+\frac12\Vq(L_t)^{1/q}
	=
	\Vq(K_t)^{1/q},
	\]
	contradicting \eqref{eq:rectangle-reverse}.
\end{proof}

\begin{remark}
	The rectangles \(K_t,L_t\) above are unconditional. Thus both this construction and the \(L_0\) counterexamples in Theorem~\ref{thm:p0-case} lie in the unconditional class, although they detect the different thresholds \(q=4\) and \(q=2\), respectively. In contrast, the counterexamples for \(0<p<1\) and \(q>2\) are non-unconditional parallelograms.
\end{remark}

\begin{proof}[Proof of Theorem~\ref{thm:sharp-threshold}]
The logarithmic case, including sharpness and equality, is
Theorem~\ref{thm:p0-case}. For $0<p\leq1$ and $0<q\leq2$,
Proposition~\ref{prop:positive-low} gives the inequality and its equality
condition. When $0<p<1$ and $q>2$,
Theorem~\ref{thm:counterexample-q-greater-than-two} gives strict midpoint
counterexamples. For $p=1$, the remaining positive range and equality
cases follow from Theorems~\ref{thm:minkowski-interior}
and~\ref{thm:q4-equality}, while Proposition~\ref{prop:q-greater-4}
excludes every $q>4$. These alternatives exhaust the parameter range.
\end{proof}

\noindent\textbf{Declaration on the use of AI.}
The proof of the equality case in Theorem~\ref{thm:minkowski-interior}
was developed through iterative discussions between the author and
ChatGPT (GPT-6 Astra Pro). The author subsequently checked and revised
the argument and takes full responsibility for the mathematical content
and the final text.


\vskip 20pt
\begin{thebibliography}{99}
% Keep bibliography hyperlinks active while displaying their text in black.
\hypersetup{urlcolor=black}

\bibitem{BLYZ}
K. J. B\"or\"oczky, E. Lutwak, D. Yang, and G. Zhang,
\emph{\href{https://doi.org/10.1016/j.aim.2012.07.015}
{The log-Brunn--Minkowski inequality}},
Adv. Math. \textbf{231} (2012), no. 3--4, 1974--1997.

\bibitem{CTW}
X. Chen, S. Tang, and S. Wang,
\emph{$L_p$ Minkowski problem and Brunn--Minkowski inequality for dual quermassintegrals},
preprint, \href{https://arxiv.org/abs/2605.28262}{arXiv:2605.28262} (2026).

\bibitem{CTWweighted}
X. Chen, S. Tang, and S. Wang,
\emph{$L_p$ Brunn--Minkowski inequality for weighted dual quermassintegrals},
preprint, \href{https://arxiv.org/abs/2607.00567}{arXiv:2607.00567} (2026).

\bibitem{Firey1962}
W. J. Firey,
\emph{\href{https://doi.org/10.7146/math.scand.a-10510}{$p$-Means of convex bodies}},
Math. Scand. \textbf{10} (1962), 17--24.

\bibitem{Gardner}
R. J. Gardner,
\emph{\href{https://doi.org/10.1017/CBO9781107341029}{Geometric Tomography}},
2nd ed., Cambridge University Press, Cambridge, 2006.

\bibitem{Hadwiger1956}
H. Hadwiger,
\emph{\href{https://doi.org/10.1007/BF02564348}
{Konkave Eik\"orperfunktionale und h\"ohere Tr\"agheitsmomente}},
Comment. Math. Helv. \textbf{30} (1956), 285--296.

\bibitem{HI}
Y. Hu and M. N. Ivaki,
\emph{Weighted centro-affine Poincar\'e inequalities},
preprint, \href{https://arxiv.org/abs/2606.04774}{arXiv:2606.04774} (2026).

\bibitem{HLYZ}
Y. Huang, E. Lutwak, D. Yang, and G. Zhang,
\emph{\href{https://doi.org/10.1007/s11511-016-0140-6}
{Geometric measures in the dual Brunn--Minkowski theory and their associated Minkowski problems}},
Acta Math. \textbf{216} (2016), no. 2, 325--388.

\bibitem{HYZ}
Y. Huang, D. Yang, and G. Zhang,
\emph{\href{https://doi.org/10.1090/bull/1858}{Minkowski problems for geometric measures}},
Bull. Amer. Math. Soc. (N.S.) \textbf{62} (2025), no. 3, 359--425.

\bibitem{KM}
A. V. Kolesnikov and E. Milman,
\emph{\href{https://doi.org/10.1090/memo/1360}
{Local $L_p$-Brunn--Minkowski inequalities for $p<1$}},
Mem. Amer. Math. Soc. \textbf{277} (2022), no. 1360.

\bibitem{LiWanZhang}
H. Li, Y. Wan, and Y. Zhang,
\emph{On the Brunn--Minkowski inequality for $q$-th dual quermassintegrals with $q>n$},
preprint, \href{https://arxiv.org/abs/2608.03949}{arXiv:2608.03949} (2026).

\bibitem{Lutwak1975}
E. Lutwak,
\emph{\href{https://doi.org/10.2140/pjm.1975.58.531}{Dual mixed volumes}},
Pacific J. Math. \textbf{58} (1975), no. 2, 531--538.

\bibitem{SadovskyZhang}
S. Sadovsky and G. Zhang,
\emph{\href{https://doi.org/10.1016/j.aim.2025.110456}
{Brunn--Minkowski and reverse isoperimetric inequalities for dual quermassintegrals}},
Adv. Math. \textbf{480} (2025), Paper No. 110456.

\bibitem{Schneider}
R. Schneider,
\emph{\href{https://doi.org/10.1017/CBO9781139003858}
{Convex Bodies: The Brunn--Minkowski Theory}},
2nd expanded ed., Cambridge University Press, Cambridge, 2014.

\bibitem{XiZhang}
D. Xi and Z. Zhang,
\emph{\href{https://doi.org/10.1090/proc/15952}
{The $L_p$ Brunn--Minkowski inequalities for dual quermassintegrals}},
Proc. Amer. Math. Soc. \textbf{150} (2022), no. 7, 3075--3086.

\bibitem{XY}
G. Xiong and K. Yang,
\emph{The uniqueness and non-uniqueness of solutions to the even dual Minkowski problem},
preprint, \href{https://arxiv.org/abs/2608.01912}{arXiv:2608.01912} (2026).

\end{thebibliography}
\end{document}